%% file: G_machines.tex
\documentclass[letterpaper]{article}
\usepackage{graphicx}
\usepackage{amsmath,amsthm}
\usepackage{amssymb}
\usepackage{float}
\usepackage[english]{babel}

\usepackage{graphicx}
\usepackage{caption}
\usepackage{subcaption}
\usepackage[linesnumbered]{algorithm2e}

\makeatletter

\newskip\@bigflushglue \@bigflushglue = -100pt plus 1fil

\def\bigcentering{\let\\\@centercr\rightskip\@bigflushglue%
\leftskip\@bigflushglue
\parindent\z@\parfillskip\z@skip}

\makeatother

\usepackage{tikz}
\usetikzlibrary{patterns,positioning,arrows,decorations.markings,calc,decorations.pathmorphing,decorations.pathreplacing}

\definecolor{rouge}{RGB}{255,77,77}
\definecolor{vert}{RGB}{0,178,102}
\definecolor{jaune}{RGB}{255,255,0}
\definecolor{violet}{RGB}{208,32,144}
\definecolor{orange}{RGB}{255,140,0}
\definecolor{bleu}{RGB}{0,0,205}

\newcommand{\lettre}[3]{
\draw[fill=#3] (#1,#2) rectangle (#1+1,#2+1);
}
\newcommand{\brouge}{\vbox to 7pt{\hbox to 13pt{
\begin{tikzpicture}[scale=0.6]
\draw [fill=rouge] (5,1) rectangle (5.5,1.5);
\end{tikzpicture}
}}}

\newcommand{\bblanc}{\vbox to 7pt{\hbox to 13pt{
\begin{tikzpicture}[scale=0.6]
\draw (5,1) rectangle (5.5,1.5);
\end{tikzpicture}
}}}

\newcommand{\bnoir}{\vbox to 7pt{\hbox to 13pt{
\begin{tikzpicture}[scale=0.6]
\draw [fill=black] (5,1) rectangle (5.5,1.5);
\end{tikzpicture}
}}}

\newcommand{\bBR}{\vbox to 12pt{\hbox to 13pt{
\begin{tikzpicture}[scale=0.6]
\draw (5,1) rectangle (5.5,1.5);
\draw[fill=rouge] (5,0.5) rectangle (5.5,1);
\end{tikzpicture}
}}}

\newcommand{\bNR}{\vbox to 12pt{\hbox to 13pt{
\begin{tikzpicture}[scale=0.6]
\draw[fill=black] (5,1) rectangle (5.5,1.5);
\draw[fill=rouge] (5,0.5) rectangle (5.5,1);
\end{tikzpicture}
}}}

\newcommand{\bRB}{\vbox to 12pt{\hbox to 13pt{
\begin{tikzpicture}[scale=0.6]
\draw[fill=rouge] (5,1) rectangle (5.5,1.5);
\draw (5,0.5) rectangle (5.5,1);
\end{tikzpicture}
}}}

\newcommand{\bRN}{\vbox to 12pt{\hbox to 13pt{
\begin{tikzpicture}[scale=0.6]
\draw[fill=rouge] (5,1) rectangle (5.5,1.5);
\draw[fill=black] (5,0.5) rectangle (5.5,1);
\end{tikzpicture}
}}}

\newcommand{\tuile}{
\begin{tikzpicture}[scale=0.4]

\draw (0,0) -- (0,2) -- (0,2) to [controls=+(45:1.5) and +(135:1.5)] (4,2) -- (4,0) to [controls=+(135:0.75) and +(45:0.75)] (2,0) to [controls=+(135:0.75) and +(45:0.75)] (0,0) -- cycle ;
\draw[fill = white] (0,0) circle (0.3);
\draw (0,0) node {{\tiny $1$}};
\draw[fill = white] (0,2) circle (0.3);
\draw (0,2) node {{\tiny $0$}};
\draw[fill = white] (2,0) circle (0.3);
\draw (2,0) node {{\tiny $1$}};
\draw[fill = white] (4,0) circle (0.3);
\draw (4,0) node {{\tiny $0$}};
\draw[fill = white] (4,2) circle (0.3);
\draw (4,2) node {{\tiny $1$}};
\draw (0,2) node[above left]{\small${1_G}$};
\draw (4,2) node[above right]{\small$a$};
\draw (0,0) node[below left]{\small$b$};
\draw (2,0) node[below]{\small${ba}$};
\draw (4,0) node[below right]{\small${ab}{ = }{ba^2}$};
\end{tikzpicture}
}

\def\ZZ{\mathbb{Z}}
\def\FF{\mathbb{F}}
\def\NN{\mathbb{N}}

\def\CC{\mathcal{C}}
\def\OO{\mathcal{O}}

\def\ag{\mathcal{A}}
\def\bg{\mathcal{B}}

\def\FF{{\mathcal{F}}}
\def\WP{{\texttt{WP}}}

\newcommand{\define}[1]{\emph{#1}}

\theoremstyle{plain}

\newtheorem*{theorem*}{Theorem}
\newtheorem*{claim*}{Claim}
\newtheorem{theorem}{Theorem}[section]
\newtheorem{lemma}[theorem]{Lemma}
\newtheorem{proposition}[theorem]{Proposition}
\newtheorem{corollary}[theorem]{Corollary}

\newtheorem*{questions*}{Questions}

\newenvironment{definition}[1][Definition]{\begin{trivlist}
\item[\hskip \labelsep {\bfseries #1}]}{\end{trivlist}}
\newenvironment{example}[1][Example]{\begin{trivlist}
\item[\hskip \labelsep {\bfseries #1}]}{\end{trivlist}}

\title{A notion of effectiveness for subshifts on finitely generated groups.}
\date{}

\author{Nathalie Aubrun\thanks{LIP, ENS de Lyon -- CNRS -- INRIA -- UCBL -- Universit\'e de Lyon}~, Sebasti\'an Barbieri\footnotemark[1]\\ and Mathieu Sablik\thanks{Aix-Marseille Universit\'e, CNRS, Centrale Marseille, I2M UMR 7373 }~\\
\texttt{nathalie.aubrun@ens-lyon.fr}\\ \texttt{sebastian.barbieri@ens-lyon.fr}\\ \texttt{mathieu.sablik@univ-amu.fr}}

\begin{document}

\maketitle

\begin{abstract}
We generalize the classical definition of effectively closed subshift to finitely generated groups. We study classical stability properties of this class and then extend this notion by allowing the usage of an oracle to the word problem of a group. This new class of subshifts forms a conjugacy class that contains all sofic subshifts. Motivated by the question of whether there exists a group where the class of sofic subshifts coincides with that of effective subshifts, we show that the inclusion is strict for several groups, including recursively presented groups with undecidable word problem, amenable groups and groups with more than two ends. We also provide an extended model of Turing machine which uses the group itself as a tape and characterizes our extended notion of effectiveness. As applications of these machines we prove that the origin constrained domino problem is undecidable for any group of the form $G \times \ZZ$ subject to a technical condition on $G$ and we present a simulation theorem which 
is 
valid in any finitely generated group.
\end{abstract}

\section*{Introduction}
\let\thefootnote\relax\footnote{{\bf Keywords:} 37B10 Symbolic dynamics, 03D10 Turing machines, 
20F10 Word problems.}

Symbolic dynamics were originally defined on $\ZZ$ in the highly influential article of Morse and Hedlund~\cite{MorseHedlund1938} in order to study discretization of dynamical systems. The main object in this theory is the subshift, that is, a set of colorings of a group by a finite alphabet which is defined by a set of forbidden patterns. In the case of the group $\ZZ^d $ with $d\geq 2$, it turns out that subshifts enjoy interesting computational properties, among which is the undecidability of the domino problem~\cite{Berger1966,Robinson1971}. Said otherwise, there is no general algorithm deciding if there exists a coloring which avoids a finite set of patterns.
This problem can be naturally generalized to any finitely generated group, nevertheless no characterization of the groups where the domino problem is undecidable is yet known, even if some partial results have arisen~\cite{Robinson1978,AubrunKari2013,BallierStein2013}. 


More recently, the use of computability theory has become essential in the study of subshifts of finite type (SFT), those defined by a finite set of forbidden patterns. For example, in $\ZZ^d$ for $d \geq 2$ the possible entropies of SFTs are characterized as right recursively enumerable numbers~\cite{HochmanMeyerovitch2010}. This type of results comes from the possibility to encode Turing machines inside $\ZZ^d$-SFTs. The study of such results led to introduce the class of effectively closed $\ZZ^d$-subshifts, defined by a recursively enumerable set of forbidden patterns. This class was introduced by Hochman~\cite{Hochman2009b} who showed that they admit an almost trivial isometric extension which is a subaction of a $\ZZ^{d+2}$-SFT. The construction was improved with two different techniques~\cite{AubrunSablik2010,DBLP:conf/birthday/DurandRS10} to get a realization in sofic $\ZZ^{d+1}$-subshifts as projective subdynamics. Thus with an increase of one of the dimension, effectively closed $\ZZ^d$-subshifts 
are very close to sofic subshifts. Hochman's result suggests that if we play with the structure on which subshifts are defined, some strong links between sofic and effectively closed subshifts may emerge. 

In this direction we investigate subshifts defined on infinite finitely generated groups and define a generalized notion of effectiveness. The difficulty for this task relies on the possibility, even for a finitely presented group, to have an undecidable word problem~\cite{novikov1955,Boone1958} -- no algorithm can decide whether a word on the generators and their inverses represents the identity element. We study the restrictions of this class with respect to this problem and define an extended definition of effectiveness by allowing the usage of oracles to the word problem of the group. 


The paper is organized as follows. Section~\ref{section.generalities} presents notations and basic notions from group theory and symbolic dynamics on finitely generated groups. In Section~\ref{section.effectiveness} we introduce a general model for effectively closed subshifts based on pattern codings and study its properties. We show that this class can be defined either by recursively enumerable or decidable sets of pattern codings, that it contains all subshifts of finite type and that it is stable under finite intersections. We also show that under the assumption that the underlying group is recursively presented this class can be defined using a maximal sets of pattern codings, it is stable under factors, finite unions and projective subdynamics. Therefore showing that this class contains all sofic subshifts and that the property of being effectively closed is a conjugacy invariant. In order to express the limitations of this class even when the group is recursively presented we introduce in 
Subsection~\ref{subsection_one_ore_less} the one-or-less-subshift $X_{\leq 1}$ which has the property of being effectively closed in recursively presented groups if and only if the word problem is decidable. This example, besides illustrating the limitations of the notion of effectively closed subshifts, answers an open question posed by Dahmani and Yaman~\cite{DahmaniYaman2008, DahmaniPC,YamanPC}. In Subsection~\ref{subsection.G_effectiveness} we briefly introduce $G$-effectively closed subshifts --subshifts which are defined by Turing machines with access to an oracle of the word problem of the group-- and list its properties. We also show that while this is a good theoretical frame in many aspects, it does not behave well with respect to projective subactions. We end Section~\ref{section.effectiveness} by studying the following question: Is there a group $G$ where the class of effectively closed subshifts coincides with the class of sofic subshifts? This question is motivated by the novel work in~\cite{
aubrun_sablik_2016} where they show that this property is held for structures resembling subshifts defined in shears of the Baumslag-Solitar group $BS(1,2)$ under the assumption of a technical property. While their result is certainly quite specific, it raises the previous question in a natural way. We give a negative answer to that question for three classes of groups, namely:

\begin{itemize}
\item recursively presented groups with undecidable word problem -- Theorem~\ref{theorem.one_or_less_non_sofic},
\item infinite amenable groups -- Theorem~\ref{theorem.amenable_sym_non_sofic},
\item groups which have two or more ends -- Theorem~\ref{theorem.more_two_ends_stricly_sofic}.
\end{itemize}

In Section~\ref{section.G_machines} we introduce an abstract model of Turing machine which instead of a bi-infinite tape uses a group. These machines are quite similar to Turing machines except that they move using a finite set of generators of $G$ and work over patterns instead of words. This object allows us to define recursively enumerable and decidable sets of patterns and gives a way to construct explicitly Turing machines with oracles. In Theorems~\ref{theorem_g_effective_is_oracle} and~\ref{theorem_oracle_is_G_effective} we make this relationship explicit with the aims of concluding in Corollary~\ref{the_great_corollary} that these machines give an alternative definition of $G$-effectively closed subshifts by $G$-machines. We end Section~\ref{section.G_machines} by giving two applications of these objects: In Theorem~\ref{theorem_group_undecidable_DP} we show that if a group $G$ satisfies that $X_{\leq 1}$ is sofic then the origin constrained domino problem of $G \times \ZZ$ is undecidable. We also 
show that this implies that the domino problem for $(G \times \ZZ)\ast H$ is undecidable for any non-trivial finite group $H$. In Theorem~\ref{Teorema_simulacion} we show that for every infinite and finitely generated group $G$ there exists a universal subshift $U$ defined over $G \times \ZZ$ such that the product of $U$ with a $G \times \ZZ$-full shift can be restricted by a finite amount of forbidden patterns and a factor code to obtain any $G$-effectively closed subshift as a projective subdynamics.

\section{Preliminaries and notation}
\label{section.generalities}

We assume from the reader basic knowledge about group theory and group presentations, a good reference is \cite{ceccherini-SilbersteinC09}. For a group $G$ we denote by $1_G$ its identity element. In this article we consider only finitely generated groups, and we denote by $S \subset G$ an arbitrary finite set of generators which is closed by inverses and contains the identity. If two words $w_1,w_2$ in $S^*$ represent the same element in $G$ we write $w_1 =_G w_2$. The undirected right \define{Cayley graph} of $G$ given by $S$, denoted by $\Gamma(G,S)$, is a vertex transitive graph whose vertices are elements of $G$ and $\{g,h\} \in \binom{G}{2}$ form an edge if there is $s \in S$ such that $gs = h$. For $g \in G$ we denote $|g|$ the length of the shortest path from $1_G$ to $g$ in $\Gamma(G,S)$. We also denote the \define{ball of size $n \geq 0$} in $\Gamma(G,S)$ as $B_n = \{g \in G \mid |g| \leq n \}$.
Naturally, the definitions above depend on the choice of generating set $S$, nevertheless all the metrics generated by the distances in such a Cayley graph are equivalent.


The \define{word problem} of $G$ is defined as the formal language:
$$\WP(G)=\left\{ w\in S^* \mid w=_G 1_G\right\}.$$
It can be shown that the decidability of the word problem is independent of the choice of generating set $S$, thus the notation $\WP(G)$ is appropriate. A fundamental result of Novikov~\cite{novikov1955} and Boone~\cite{Boone1958} exhibits finitely presented groups with undecidable word problem.



Let $\ag$ be a finite alphabet. We say that the set $\ag^G = \{ x: G \to \ag\}$ equipped with the left group action $\sigma: G \times \ag^G \to \ag^G$ such that $(\sigma_g(x))_h = x_{g^{-1}h}$ is a \textit{full shift}. The elements $a \in \ag$ and $x \in \ag^G$ are called \define{symbols} and \define{configurations} respectively. With the discrete topology on $\ag$, the product topology in $\ag^G$ is compact. This topology is generated by a clopen subbasis given by the \define{cylinders} $[a]_g = \{x \in \ag^G | x_g = a\in \ag\}$. Since $G$ is countable, $\ag^G$ is metrizable and an ultrametric which generates the product topology is given by $\displaystyle{d(x,y) = 2^{-\inf\{|g|\; \mid\; g \in G:\; x_g \neq y_g\}}}.$ A \emph{support} is a finite subset $F \subset G$. A \emph{pattern with support $F$} is an element $p$ of $\ag^F$. We denote the set of finite patterns by $\ag_G^* := \bigcup_{F \subset G, |F| < \infty}{\ag^F}$. For $p \in \ag^F$ and $g \in G$ the \define{cylinder generated by $p$ on $g$} is $[
p]_g := \bigcap_{h \in F}[p_h]_{gh}$. 

\begin{definition}
A subset $X$ of $\ag^G$ is a \define{subshift} if it is $\sigma$-invariant -- $\sigma(X)\subset X$ -- and closed for the cylinder topology. Equivalently, $X$ is a subshift if there exists $\FF \subset \ag_G^*$ such that: $$ X = X_{\FF} =: \bigcap_{p \in \FF, g \in G} \ag^G \setminus [p]_{g}.$$\end{definition}


Let $X,Y$ be two subshifts over alphabets $\ag_X,\ag_Y$. We call a continuous $G$-equivariant -- i.e. $\sigma$-commuting -- function $\phi: X \to Y$ a \define{morphism}. A famous theorem by Curtis, Lyndon and Hedlund -- see for example  ~\cite{ceccherini-SilbersteinC09} -- gives a combinatorial characterization of morphisms as block codes: namely, $\phi$ is a morphism if and only if there exists a finite $F \subset G$ and a local function $\Phi: \ag_X^F \to \ag_Y$ such that $\phi(x)_g := \Phi(\sigma_{g^{-1}}(x)|_{F})$. We say $\phi$ is a \define{factor} if $\phi$ is surjective, and a \define{conjugacy}  if it is bijective. Whenever there is a factor code $\phi : X \to Y$ we write $X \twoheadrightarrow Y$ and say that $Y$ is a \textit{factor} of $X$ and that $X$ is an \textit{extension} of $Y$. Furthermore, if $\phi$ is a conjugacy we will write $X \simeq Y$ and say they are \textit{conjugated}. The conjugacy is an equivalence relation which preserves most of the topological dynamics of a system.

%
%
%

We say that a subshift $X \subset \ag^G$ is of \define{finite type} -- SFT for short -- if it can be defined by a finite set of forbidden patterns, that is, $|\FF| < \infty$ and $X = X_{\FF}$. We say that $X$ is \define{sofic} if there exists an SFT $Y$ and a factor code $\phi: Y \twoheadrightarrow X$. The class of sofic subshifts is the smallest class closed under factor codes that contains every SFT. Both classes are conjugacy invariants, that is, the property of belonging to them is preserved under conjugacy.

\section{Effectiveness on finitely generated groups}
\label{section.effectiveness}
When $G = \ZZ$, patterns can be identified as words over a finite alphabet. We say a subshift $X \subset \ag^{\ZZ}$ is \define{effectively closed} if there is a recursively enumerable set of forbidden words that defines it. We intend to generalize this definition to the class of finitely generated groups. On~$\ZZ^d$, a finite pattern is no longer a word, but it can be easily coded as a word -- via any recursive bijection between $\ZZ^d$ and $\ZZ$. Then effective $\ZZ^d$-subshifts correspond to subshifts which can be defined by a set of forbidden patterns that admits a recognizable set of codings. In groups with undecidable word problem this recursive bijection does not exist.

In this section we first take the previous ideas of codings to the context of finitely generated groups by introducing the formalism of pattern codings and explore the limitations of this concept when the word problem of the group is not decidable or recursively enumerable. At this point we introduce the subshift $X_{\leq 1}$ which consists in all configurations containing at most one appearance of a non-zero symbol, and use it to exemplify these previous constraints. Next we extend the notion of effectiveness by adding the power of an oracle to $\WP(G)$. We remark the stability properties for this extended class and compare them with sofic subshifts and SFTs. Finally we exhibit three big classes of groups where this class does not coincide with the one of sofic subshifts.

\subsection{Classical effectiveness}
\label{subsection.Z_effectiveness}

Let $G$ be a finitely generated group and $\ag$ an alphabet. A \define{pattern coding} $c$ is a finite set of tuples $c=(w_i,a_i)_{i \in I}$ where $w_i \in S^{*}$ and $a_i \in \ag$. We say that a pattern coding is \define{consistent} if for every pair of tuples such that $w_i =_G w_j$ then $a_i = a_j$. For a consistent pattern coding $c$ we define the pattern $p(c) \in \ag^F$ where $F= \bigcup_{i \in I}w_i$ and $p(c)_{w_i}=a_i$.

\begin{example}
 Let $BS(1,2) \cong \langle a,b\mid ab=ba^2\rangle$ be a Baumslag-Solitar group and $\ag=\{0,1\}$. Then the pattern coding
 $$
 \begin{array}{ccccc}
 (\epsilon,0) & ~ & (b,1) & ~ & (a,1)\\
 (ab,0) & & (ba^2,0)& & (ba,1)\\
 \end{array}
 $$
 is consistent, since all the words above on $S = \{a,b,a^{-1},b^{-1}\}$ represent different elements in $G$ except for $ab$ and $ba^2$ that are assigned the same symbol. The pattern $p$ it defines is: 
 \begin{center}
 \tuile
 \end{center}
 But the pattern coding 
 $$
 \begin{array}{cccccccc}
 (\epsilon,0) & ~ & (a^2,1) & ~ & (bab^{-1}a, 1) \\
 (a,1) & & (ba,1) & & (abab^{-1},0) \\
 \end{array}
 $$
 is inconsistent since words $abab^{-1}$ and $bab^{-1}a$ represent the same element in $G$ but are assigned different symbols.
\end{example}

A set of pattern codings $\CC$ is said to be recursively enumerable if there is a Turing machine which takes as input a pattern coding $c$ and accepts it if and only if $c \in \CC$.

\begin{definition}\label{definition_effectively_closed}
	A subshift $X \subset \ag^G$ is \define{effectively closed} if there is a recursively enumerable set of pattern codings $\CC$ such that:$$ X = X_{\CC} := \bigcap_{g \in G, c \in \CC} \left( \ag^G \setminus \bigcap_{(w,a) \in c}[a]_{gw} \right).$$
\end{definition}

The specific choice of the set of generators $S$ is irrelevant as one can easily translate one in terms of the other. Notice that if a pattern coding is inconsistent then $\bigcap_{(w,a) \in c}[a]_w = \emptyset$ and that if it is consistent then $\bigcap_{(w,a) \in c}[a]_w = [p(c)]_{1_G}$. Therefore, the subshift defined by a set of pattern codings $\CC$ only depends on the set of consistent ones, in the sense that if $p(\CC)$ is the set of patterns defined by the consistent pattern codings of $\CC$ then $X_{\CC} = X_{p(\CC)}$.

We could also define this class by the existence of a decidable family rather than a recursively enumerable one. This justifies the usage of the word ``effectively''. The following proposition is commonly known to hold true in $G = \ZZ^d$. Here we present a general version which works in every finitely generated group.

\begin{proposition}
	\label{proposition.F_decidable_Z_decidable}
	Let $X \subset \ag^G$ be an effectively closed subshift. Then there exists a decidable set of pattern codings $\CC$ such that $X = X_{\CC}$.
\end{proposition}

\begin{proof}
	Let $\CC'$ a recursively enumerable set of pattern codings such that $X = X_{\CC'}$. If $\CC'$ is finite the result is trivial. Otherwise there exists a recursive enumeration $\CC' = \{c_0,c_1,\dots\}$. For a pattern coding $c$ we define its length as $|c| = \max_{(w,a) \in c}{|w|}$. For $n \in \NN$ let $L_n = \max_{k \leq n}|c_k|$ and define $\CC_n$ as the finite set of all pattern codings $c$ which satisfy the following properties:
	\begin{itemize}
		\item Every $w \in S^*$ with $|w|\leq L_n$ appears in exactly one pair in $c$.
		\item $(w,a) \in c$ implies that $|w|\leq L_n$.
		\item If $(w,a) \in c_n$ then $(w,a) \in c$.
	\end{itemize}
	That is, $\CC_n$ is the set of all pattern codings which are completions of $c_n$ up to every word of length at most $L_n$ in every possible way. Consider $\CC = \bigcup_{n \in \NN}\CC_n$. Clearly it satisfies that $X = X_{\CC}$. We claim it is decidable.
	
	Consider the algorithm which does the following on input $c$: It initializes $n$ to $0$. Then it enters into the following loop: First it produces the pattern coding $c_n$. If $L_n > |c|$ it rejects the input. Otherwise it calculates the set $\CC_n$. If $c \in \CC_n$ then it accepts, otherwise it increases the value of $n$ by $1$. 
	
	As $L_n$ is increasing and cannot stay in the same value indefinitely this algorithm eventually ends for every input. \end{proof}

In what follows we will show which are the liberties one can take when choosing a defining set of pattern codings and the structural properties of this class. Some of these are related to the following notion in group theory. A group is said to be \define{recursively presented} if there is a presentation $G \cong \langle S,R \rangle$ where $S$ is a recursive set and $R \subset S^*$ a recursively enumerable language. As we only consider finitely generated groups $S$ is always finite and thus recursive, so we take the second requirement as the definition.

\begin{proposition}\label{dumbpropositionequiv3zeffective}
Let $G$ be a finitely generated group and $\ag$ be an alphabet with at least two symbols. The following are equivalent:

\begin{enumerate}
\item $G$ is recursively presented.
\item $\WP(G)$ is recursively enumerable.
\item The set of inconsistent pattern codings is recursively enumerable.
\end{enumerate}

\end{proposition}

\begin{proof}
The equivalence between the two first statements is trivial. Let $G$ have recursively enumerable word problem. As $u =_G v \Leftrightarrow uv^{-1} =_G 1_G$ the set of inconsistent pattern codings is recursively enumerable. Indeed, for $n\in \NN$, a Turing machine on entry $c$ can simulate iteratively for $n$ steps the machine recognizing $\WP(G)$ applied to $uv^{-1}$ for every pair $(u,a),(v,b) \in c$ with $a \neq b \in \ag$ and accept if this procedure accepts for some $n$. Conversely, given $w\in S^*$, it suffices to give as input to the machine recognizing the inconsistency of the pattern codings $c = \{(\epsilon,a),(w,b)\}$ with $a \neq b \in \ag$ in order to recognize if $w =_G 1_G$.
\end{proof}

\begin{lemma}
\label{lemma.F_maximal_Z_effective}
Let $X \subset \ag$ be an effectively closed subshift. If $G$ is recursively presented then it is possible to choose $\CC$ to be a recursively enumerable and maximal -- for inclusion -- set of pattern codings such that $X = X_{\CC}$.
\end{lemma}

This lemma is fundamental is the rest of the article. Indeed, every time the statement of a result requires as hypothesis that a group $G$ is recursively presented, this is because its proof uses the existence of a recursively enumerable and maximal set of pattern codings for some $G$-subshift.

\begin{proof}
	A pattern coding $c$ belongs to the maximal set $\CC$ defining $X$ if and only if $X \cap \bigcap_{(w,a) \in c}[a]_{w} = \emptyset$. Let $c \in \CC$ and $\CC'$ a recursively enumerable set such that $X = X_{\CC'}$. Then: $$\bigcap_{(w,a) \in c}[a]_{w} \subset \bigcup_{c' \in \CC', g \in G }\bigcap_{(w',a') \in c'}[a']_{gw'}.$$
	By compactness we may extract a finite open cover indexed by $c'_i,g_i$ such that:
	\begin{equation}\tag{$\star$}\label{equation1}
	\bigcap_{(w,a) \in c}[a]_{w} \subset \bigcup_{i \leq n}\bigcap_{(w',a') \in c'_i}[a']_{g_iw'}
	\end{equation}
	
	Note that each of these $g_i$ can be seen as a finite word in $S^*$. Now let $T$ be the Turing machine which does iteratively for $n \in \NN$ the following:
	\begin{itemize}
		\item Runs $n$ steps the machine $T_1$ recognizing $\WP(G)$ for every word in $S^*$ of length smaller than $n$.
		\item Runs $n$ steps the machine $T_2$ recognizing $\CC'$ for every pattern coding defined on a subset of words of $S^*$ of length smaller than $n$.
		\item Let $\sim_n$ be the equivalence relation for words in $S^*$ of length smaller than $n$ such that $u \sim_n v$ if $uv^{-1}$ has been already accepted by $T_1$. Let $\CC_n$ be the pattern codings already accepted by $T_2$. If every word in $c$ has length smaller than $n$ check if the following relation is true under $\sim_n$:
		
		$$\bigcap_{(w,a) \in c}[a]_{w} \subset \bigcup_{c' \in \CC_n, |u| \leq n}\bigcap_{(w',a') \in c'}[a']_{uw'}$$
		
		If it is true, accept, otherwise increase $n$ by $1$ and continue.	\end{itemize}
	Let $m$ be the max of all $|w|$ such that $(w,a)\in c$, and $|w'|$ such that $(w',a')\in c'_i$ and all $|g_i|$. By definition, there exists an $N \in \NN$ such that every $c'_i$ for $i \leq n$ is accepted and every word representing $1_G$ of length smaller than $2m$ is accepted. This means that at stage $N$ relation~\ref{equation1} is satisfied and $T$ accepts $c$. If $c$ is not in the maximal set, the machine never accepts.\end{proof}

Lemma~\ref{lemma.F_maximal_Z_effective} is no longer true if $G$ is not recursively presented. Indeed, the maximal set of pattern codings defining the full shift is given by the set of all inconsistent pattern codings, which is recursively enumerable if and only if $G$ is recursively presented by Proposition~\ref{dumbpropositionequiv3zeffective}. 

\begin{proposition}
	The class of SFTs is contained in the class of effectively closed subshifts.
\end{proposition}

\begin{proof}
	Let $X$ be an SFT. Then $X = X_{\FF}$ for a finite set $\FF$. For each $p \in \FF$ consider a pattern coding $c_p$ such that $p(c_p) = p$ and let $\CC = \{ c_p \mid p \in \FF\}$. Clearly $X = X_{\CC}$ and as $\CC$ is finite it is recursively enumerable.\end{proof}

\begin{proposition}\label{proposition_stability_intersection}
	The class of effectively closed subshifts is closed by finite intersections.
\end{proposition}

\begin{proof}
	Let $X = X_{\CC_X}$ and $Y = Y_{\CC_Y}$ be effectively closed subshifts. Without loss of generality suppose $X,Y \subset \ag^G$ (same alphabet) and note that:
	\begin{align*}
	X \cap Y & = \left( \ag^G \setminus \bigcup_{g \in G, c \in \CC_X} \bigcap_{(w,a) \in c}[a]_{gw} \right) \cap \left( \ag^G \setminus \bigcup_{g \in G, c \in \CC_Y} \bigcap_{(w,a) \in c}[a]_{gw} \right) \\
	& = \ag^G \setminus \bigcup_{g \in G, c \in \CC_X \cup \CC_Y} \bigcap_{(w,a) \in c}[a]_{gw}\\
	& = X_{\CC_X \cup \CC_Y}
	\end{align*}
	Therefore, it suffices on entry $c$ to launch the Turing machines recognizing $\CC_X$ and $\CC_Y$ in parallel and accept if either of them accepts.
\end{proof}

The result obviously does not extend to countable intersections. If it were so, since every possible subshift is obtainable as an intersection of SFTs (enumerate the forbidden patterns, define $X_n = X_{p_1,\dots,p_n}$, then $X = \bigcap_{n \in \NN}X_n$), we would conclude that all subshifts are effectively closed. But there is an uncountable number of subshifts on a fixed alphabet, and effectively closed subshifts clearly constitute a countable set, so there must be one that is not effectively closed. 

\begin{proposition}\label{proposition_stability_union}
	For a recursively presented group the class of effectively closed subshifts is closed by finite unions.
\end{proposition}

\begin{proof}
Let $X = X_{\CC_X}$ and $Y = Y_{\CC_Y}$ be effectively closed subshifts. As $G$ is recursively presented we can suppose $\CC_X$ and $\CC_Y$ are maximal as in Lemma~\ref{lemma.F_maximal_Z_effective} As in the previous proof we can show: $$X \cup Y =\ag^G \setminus \left( \left(\bigcup_{g \in G, c \in \CC_X } \bigcap_{(w,a) \in c}[a]_{gw}\right) \cap \left(\bigcup_{g \in G, c \in  \CC_Y} \bigcap_{(w,a) \in c}[a]_{gw} \right)  \right) $$

Thus, as these sets are maximal we have  $X \cup Y = X_{\CC_X \cap \CC_Y}$. It suffices therefore to launch both Turing machines and accept if both accept.\end{proof}

\begin{proposition}\label{effectiveness_closed_factors}
 For recursively presented groups the class of effectively closed subshifts is closed under factors. 
\end{proposition}

\begin{proof}
	Let $X \subset \ag_X^G$ be an effectively closed subshift. As $G$ is recursively presented, the recursively enumerable set of pattern codings $\CC_X$ can be chosen to be maximal by Lemma~\ref{lemma.F_maximal_Z_effective}. Consider a factor code $\phi: X \twoheadrightarrow Y$ defined by a local function $\Phi: \ag_X^F \to \ag_Y$. Let $f_1,\dots, f_{|F|}$ be words in $S^*$ such that $F = \{f_1,\dots, f_{|F|}\}$.
	
	As $\phi$ is surjective, for each $a \in \ag_Y$ then $|\Phi^{-1}(a)| > 0$. Therefore we can associate to a pair $(w,a)$ a non-empty finite set of pattern codings $$\CC_{w,a} = \{ (wf_i, p_{f_i})_{i = 1,\dots,|F|} \mid p \in \Phi^{-1}(a) \}.$$
	
	That is, $\CC_{w,a}$ is a finite set of pattern codings over $\ag_X$ representing every possible preimage of $a$. For a pattern coding $c = (w_i,a_i)_{i \leq n}$ where $a_i \in \ag_Y$ we define: 
	$$\CC_{c} = \{ \bigcup_{(w,a)\in c}\widetilde{c}_{w,a} \mid \widetilde{c}_{w,a} \in \CC_{w,a}    \}.$$
	
	That is, $\CC_c$ is the finite set of pattern codings formed by choosing one possible preimage for each letter. This set has the property that if $\Phi$ is applied pointwise then $\Phi(p(\CC_c)) = \{p(c)\}$. Let $T$ be the Turing machine which on entry $c$ runs the machine recognizing $\CC_X$ on every pattern coding in $\CC_c$. If it accepts for every input, then $T$ accepts $c$. Let $\CC_Y$ be the set of pattern codings accepted by $T$. We claim $Y = Y_{\CC_Y}$.
	
	Let $y \in  Y_{\CC_Y}$ and $n \in \NN$. For each pattern coding $c$ such that $p(c) = y|_{B_n}$, there is a pattern coding $c_n \in \CC_c$ which does not belong to $\CC_X$. As $\CC_X$ is maximal we have that $[p(c_n)] \cap X \neq \emptyset$. Extracting a configuration $x_n$ from $[p(c_n)] \cap X$ we obtain a sequence $(x_n)_{n \in \NN}$. By compactness there is a converging subsequence with limit $\widetilde{x} \in X$. By continuity of $\phi$ we have that $y = \phi(\widetilde{x}) \in Y$. Conversely if $y \in Y$ there exists $x \in X$ such that $\phi(x)=y$. Therefore for every finite $F' \subset G$ and pattern coding $c$ with $p(c) = y|_{F'}$ there exists a pattern coding $\widetilde{c} \in \CC_c$ such that $p(\widetilde{c}) = x|_{F'F}$. Therefore, $c \notin \CC_y$ and thus $y \in Y_{\CC_Y}$.
\end{proof}

\begin{corollary}
	For a recursively presented group the following are true:
	\begin{itemize}
		\item The class of effectively closed subshifts is invariant under conjugacy.
		\item The class of effectively closed subshifts contains all sofic subshifts.
	\end{itemize}
\end{corollary}

We do not know if the previous results extend to the general case where $G$ is not recursively presented. The main obstruction is that without that hypothesis there is no control on the representations of the finite set $F$ which defines the local rule of the factor. As an example, suppose $F = \{1_G\}$, that is $\Phi: \ag_X \to \ag_Y$. In order to detect forbidden patterns by using the recursively enumerable set defining $X$ we would need to touch all possible representations of $F$, which is exactly the set $\WP(G)$.

Let $H \leq G$ be a subgroup of $G$. Given a subshift $X \subset \ag^G$ the $H$-projective subdynamics of $X$ is the subshift $\pi_H(X) \subset \ag^H$ defined as:

$$ \pi_H(X)  = \{ x \in \ag^H \mid \exists y \in X, \forall h \in H, x_h = y_h\}$$

\begin{proposition}\label{proposition_subdyn_effective}
	Let $G$ be a recursively presented group and $H \leq G$ a finitely generated subgroup of $G$. If $X \subset \ag^G$ is effectively closed, then its $H$-projective subdynamics $\pi_H(X)$ is effectively closed.
\end{proposition}

\begin{proof}
 As $H$ is finitely generated, there exists a finite set $S' \subset H$ such that $\langle S' \rangle = H$. As $G$ is finitely generated by $S$ there exists a function $\gamma: S' \to S^*$ such that $s' =_G \gamma(s')$ (that is, every element of $S'$ can be written as a word in $S^*$). Extend the function $\gamma$ to act by concatenation over words in $S'^*$. 
 
 As $G$ is recursively presented, by Lemma~\ref{lemma.F_maximal_Z_effective} the set of pattern codings $\CC_G$ defining $X$ can be chosen to be maximal. Let $c = (w_i,a_i)_{i \in I}$ a pattern coding where $w_i \in S'^*$ and consider $\gamma(c) = (\gamma(w_i),a_i)_{i \in I}$. Let $T$ be the Turing machine which on entry $c$ runs the algorithm recognizing $\CC_G$ on entry $\gamma(c)$ and accepts if and only if this machine accepts. Clearly $\CC_H = \{ c \mid T \text{ accepts } c  \}$ is recursively enumerable. Also, as $\CC_G$ is a maximal set of pattern codings then $c \in \CC_H \iff [p(\gamma(c))] \cap X = \emptyset$. Therefore $\pi_H(X)  = X_{\CC_H}$.\end{proof}


Besides all of these obstructions, even for recursively presented groups there are very simple subshifts which do not fall in this class. In order to illustrate this limitation we introduce the \define{One-or-less subshift}.

\subsection{The One-or-less subshift}\label{subsection_one_ore_less}

Consider the subshift $X_{\leq 1} \subset \{0,1\}^G$ whose configurations contain at most one appearance of the letter $1$.

$$ X_{\leq 1} = \{ x \in \{0,1\}^G \mid 1 \in \{x_g,x_h\} \implies g = h \}$$

As we shall see later, this subshift is related to the word problem of a group. In the literature, it is sometimes called the ``sunny side up'' subshift. We begin by showing some properties of $X_{\leq 1}$.

\begin{proposition}
	If $G$ is infinite, then $X_{\leq 1}$ is not an SFT.
\end{proposition}

\begin{proof}
	Suppose $X_{\leq 1} = X_{\FF}$ for a finite $\FF$ and let $$F = \bigcup_{p \in \FF, p \text{ has support } F'} F'.$$ 
	As $G$ is infinite, there exists $g \in G$ such that $F \cap gF = \emptyset$. As all forbidden patterns have their support contained in $F$ then: $$([1]_g \cap [1]_{1_G}) \cap X_{\FF} \neq \emptyset$$
	But $([1]_g \cap [1]_{1_G}) \cap X_{\leq 1} = \emptyset$. Therefore $X_{\FF} \not\subset X_{\leq 1}$.\end{proof}

This subshift has already been studied in~\cite{DahmaniYaman2008}. In that article the authors showed that the action of a relatively hyperbolic group on its boundary is related to $X_{\leq 1}$ being sofic. They said a group $G$ has the \define{special symbol property} if $X_{\leq 1} \subset \{0,1\}^G$ is a sofic subshift. They furthermore proved some stability properties among which are:

\begin{enumerate}
	\item if $G$ has the special symbol property then $G$ is finitely generated.
	\item If $G$ splits in a short exact sequence $1 \to N \to G \to H \to 1$ and both $N$ and $H$ satisfy the special symbol property, then $G$ also does.
	\item If $[G:H] < \infty$ then $G$ has the special symbol property if and only if $H$ does.
	\item The special symbol property is true for:
	\begin{itemize}
		\item Finitely generated free groups.
		\item Finitely generated abelian groups.
		\item Hyperbolic groups.
		\item Poly-hyperbolic groups.
	\end{itemize}
\end{enumerate}

Besides the restriction of $G$ being finitely generated the authors did not present any example of group without the special symbol property. In this section we introduce a computability obstruction for this property which at the same time shows one of the limitations of the classical approach to effectiveness.

\begin{proposition}\label{proposition.one_or_less_non_effective}
Let $G$ be a recursively presented group. Then $X_{\leq 1}$ if effectively closed if and only if $\WP(G)$ is decidable.
\end{proposition}

\begin{proof}
If $\WP(G)$ is decidable then $X_{\leq 1}$ is effectively closed. Indeed, an algorithm recognizing a maximal set of pattern codings $\CC$ such that $X_{\leq 1} = X_{\CC}$ is the following: On input $c$ it considers every pair $(w_1,1), (w_2,1)$ in $c$ and accept if and only if $w_1w_2^{-1} \neq_G 1_G$ for a pair.
Conversely, as $G$ is recursively presented, the word problem is already recursively enumerable. It suffices to show it is co-recursively enumerable. 

By Lemma~\ref{lemma.F_maximal_Z_effective} there exists a maximal set of forbidden pattern codings $\CC$ with $X_{\leq 1} = X_{\CC}$. Given $w \in S^*$, consider the pattern coding $c_w = \{(\epsilon,1),(w,1)\}$. Note that $w \neq_G 1_G \iff c_w \in \CC$. Therefore the the algorithm which on entry $w \in S^*$ runs the algorithm recognizing $\CC$ on entry $c_w$ and accepts if and only if this one accepts, recognizes $S^* \setminus \WP(G)$. Hence $\WP(G)$ is co-recursively enumerable.\end{proof}

Using Proposition~\ref{effectiveness_closed_factors} we obtain the following corollary which answers a question of Dahmani and Yaman~\cite{DahmaniPC,YamanPC}.

\begin{corollary}\label{corollary_one_or_less_not_sofic}
	If $G$ is recursively presented and $\WP(G)$ is undecidable, then $X_{\leq 1}$ is not sofic.
\end{corollary}

\subsection{$G$-Effectiveness}
\label{subsection.G_effectiveness}

In order to escape the limitations of Lemma~\ref{lemma.F_maximal_Z_effective} and Proposition~\ref{effectiveness_closed_factors} and include subshifts such as $X_{\leq 1}$, we introduce the class of $G$-effectively closed subshifts. The aim of this subsection is to briefly introduce these objects and remark their properties. They are studied in detail in Section~\ref{section.G_machines}.

A set of pattern codings $\CC$ is said to be \define{recursively enumerable with oracle $\mathcal{O}$} if there exists a Turing machine with oracle $\OO$ which accepts on input $c$ if and only if $c \in \CC$. The oracle $\OO$ is a language to which these special machines have the right to ask if $w \in \OO$ and receive the correct answer in one step.

\begin{definition}\label{definition_G_effectively_closed}
	A subshift $X \subset \ag^G$ is \define{$G$-effectively closed} if there is a set of pattern codings $\CC$ such that $X = X_{\CC}$, and $\CC$ is recursively enumerable with oracle $\WP(G)$.
\end{definition}

We remark the following properties that either fall directly from the definition or are obtained from adding the word problem $\WP(G)$ as oracle to the previous results. Let $G$ be a finitely generated group, then:

\begin{enumerate}
	\item If $X$ a $G$-effectively closed subshift then a maximal set of pattern codings $\CC$ such that $X = X_{\CC}$ is recursively enumerable with oracle $\WP(G)$.
	\item The class of $G$-effectively closed subshift is closed under finite intersections and unions.
	\item The class of $G$-effectively closed subshifts is closed under factors.
	\item Being $G$-effectively closed is a conjugacy invariant.
	\item The class of $G$-effectively closed subshifts contains all sofic subshifts.
	\item The class of $G$-effectively closed subshifts contains all effectively closed subshifts.
	\item If $\WP(G)$ is decidable, then every $G$-effectively closed subshift is effectively closed.
	\item $X_{\leq 1}$ is a $G$-effectively closed subshift.
\end{enumerate}

The only property which does not extend nicely is the stability under taking projective subdynamics. Clearly if $X\subset \ag^G$ is $G$-effectively closed then for any finitely generated $H \leq G$ we would have that the $H$-projective subdynamics $\pi_H(X) $ can be defined by a set of pattern codings which is recursively enumerable with oracle $\WP(G)$. Nevertheless, it may not be possible to define such set with Turing machines using oracle $\WP(H)$.

\begin{proposition}\label{proposition_subdyn_not_stable}
	Let $G$ be a group which is not recursively presented. There exists a $G \times \ZZ$-effectively closed subshift $X \subset \ag^{G \times \ZZ}$ such that its $\ZZ$-projective subdynamics is not $\ZZ$-effectively closed.
\end{proposition}

\begin{proof}
	Let $\ag = S \cup \{\star\}$. For $w \in S^*$, let $p_w$ defined over the support $\{1_G\} \times \{0, \dots |w|+1\}$ such that $(p_w)_{(1_G,0)} = (p_w)_{(1_G,|w|+1)} = \star$ and for $j \in \{1,\dots,|w|\}$ then  $(p_w)_{(1_G,j)} = w_j$. Let $X := X_{\FF} \subset \ag^{G \times \ZZ}$  be defined by the set of forbidden patterns $\FF = \{p_w \mid w \in \WP(G) \}$. Clearly $X$ is $G$-effectively closed. Every $\ZZ$-coset of a configuration $x \in X$ contains a bi-infinite sequence $y \in \ag^{\ZZ}$ such that either $y$ contains at most one symbol $\star$ or every word appearing between two appearances of $\star$ represents $1_G$ in $G$.
	
	We claim that $\pi_{\ZZ}(X)$ is not effectively closed. If it were, there would exist a maximal set of forbidden pattern codings which is recursively enumerable and defines $\pi_{\ZZ}(X)$. Therefore given $w \in S^*$ a machine could run the algorithm for the word $\star w\star$ and it would be accepted if and only if $w =_G 1_G$. This would imply that $G$ is recursively presented.\end{proof}

In Section~\ref{section.G_machines} a characterization of these subshifts by Turing machines which instead of a tape have Cayley graphs of groups is given. This allows an alternative definition of $G$-effectiveness which at the same time gives a concrete construction of Turing machines with oracle.

\subsection{Groups with $G$-effective subshifts which are not sofic. }
\label{section.amenable_groups}

In the work of two of the authors~\cite{aubrun_sablik_2016}, it is shown that for subshifts in the hyperbolic plane that satisfy a technical condition, the property of being sofic is equal to the property of being effectively closed. By hyperbolic plane it is meant the monoid $M = \langle a,a^{-1},b \mid ab = ba^2, aa^{-1} = 1_M \rangle$ which looks like a shear of the Baumslag-Solitar group $BS(1,2)$ (here all the definitions given above for groups naturally extend to monoids). The reason behind this fact is that the doubling structure of this monoid allows to transmit the information on a row $b^n\langle a\rangle$ to all rows $b^{m}\langle a \rangle$ where $m \geq n$, and thus a Turing machine calculation can be implemented as an extra SFT extension. This shows that any subshift defined by a recursively enumerated set of pattern codings is in fact a sofic subshift.

This result raises the following questions:

	\begin{itemize}
		\item If we consider the group $BS(1,2)$, is it true that every effectively closed subshift is sofic?
		\item Is there any group $G$ such that every $G$-effectively closed subshift is sofic?
		\item Is there any group such that the class of effectively closed subshifts and sofic subshifts coincide?
	\end{itemize}
	
In this section we give a negative answer to the first question, and give partial negative answers to the second and third questions. More precisely, we show that the equality between the class of $G$-effectively closed subshifts and sofic subshifts cannot happen in three cases: recursively presented groups with undecidable word problem, amenable groups and groups with two or more ends.

\begin{theorem}
\label{theorem.one_or_less_non_sofic}
For every recursively presented group $G$ with undecidable word problem there exists a $G$-effectively closed subshift which is not sofic
\end{theorem}

\begin{proof}
	The subshift $X_{\leq 1}$ is $G$-effectively closed but not sofic for recursively presented $G$ as stated in Corollary~\ref{corollary_one_or_less_not_sofic}.\end{proof}

Clearly, this does not say anything about the existence of effectively closed subshifts which are not sofic in this case. In fact, it is not known whether $X_{\leq 1}$ is sofic for all groups with decidable word problem or not. 

For the case of amenable groups, we take inspiration in a classical construction for $\ZZ^2$ called the mirror shift. It consists of all configurations over the alphabet $\ag = \left\{ \bblanc,\bnoir,\brouge \right\}$ such that these forbidden patterns do not appear.

$$\FF :=\left\{\bBR,\bNR,\bRB,\bRN\right\}\cup \bigcup_{w\in \ag^*}\left\{ \brouge w \brouge, \bnoir w \brouge \tilde{w} \bblanc , \bblanc w \brouge \tilde{w} \bnoir \right\},$$

where $\tilde{w}$ denotes the reverse of the word $w$.

This subshift is easily seen to be effectively closed, while it can be proven that it is not sofic. Indeed, if $S$ is the canonical set of generators of $\ZZ^2$, then $|B_{n+1} \setminus B_n| /|B_n| $ tends to $0$ as $n$ goes to infinity. From this it is possible to deduce that in a suitable SFT extension of the mirror shift, there are two different patterns sharing the same boundary which yield different patterns in the mirror subshift. As shown in Figure~\ref{figure.example_config_mirror}, switching a pattern for the other produces a point outside the subshift yielding a contradiction. In what follows we generalize this technique to amenable groups.

\begin{figure}[H]
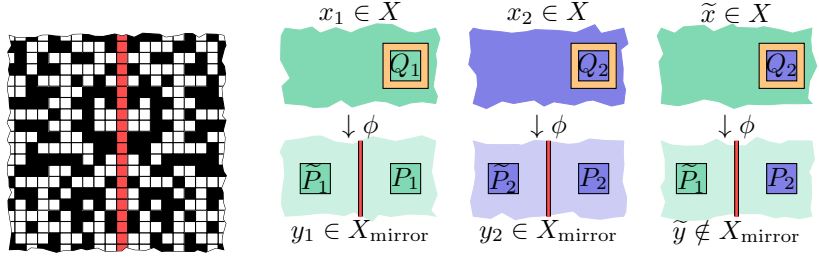

	\centering
	\include{example_config_mirror}
	\caption{Configuration in the mirror shift and technique showing non-soficity.}
	\label{figure.example_config_mirror}
\end{figure}

For a finitely generated group we say a sequence of elements $(g_n)_{n \in \NN}$ is recursive if there is a Turing machine which on input $n$ produces a word $w \in S^*$ such that $w =_G g_n$. If the Turing machine uses oracle $\OO$ then the sequence is said to be recursive with oracle $\OO$.

\begin{lemma}\label{lemma_secuencias_recursivas_disjuntas}
	For every infinite group $G$ there exist a pair of recursive sequences $(g_n)_{n \in \NN}$, $(h_n)_{n \in \NN}$ with oracle $\WP(G)$ such that the family of sets $$\mathcal{S} = \{\{1_G\}\} \cup \{g_nB_n\}_{n \in \NN} \cup \{h_nB_n\}_{n \in \NN}$$ is pairwise disjoint.
\end{lemma}

\begin{proof}
	Fix a total order on $S$ and extend it to a lexicographic order in $S^*$. Let $T_g,T_h$ be the Turing machines with oracle $\WP(G)$ that do the following on entry $n \in \NN$.
	\begin{itemize}
		\item Let $N = 1+2\sum_{k = 1}^{n}(2k+1) = 1+2n(n+2)$. Solve the word problem for every $w \in S^*$ such that $|w| \leq 2N$. This allows to construct $B_N$ of the Cayley graph $\Gamma(G,S)$.
		\item Assign the value $0$ to every $g \in B_N\setminus \{1_G\}$, and $1$ to $1_G$. Assign initially the value $g_0,\dots,g_n,h_0,\dots,h_n$ to $\epsilon$. And initiate a variable $k$ with its value set initially to $0$.
		\item While $k \leq n$ do the following: Iterate over all $w \in S^*$ lexicographically. If for $w$ all of the values of $wB_k$ have the value $0$ then: 
		\begin{itemize}
			\item Turn all of the values in $wB_k$ to $1$.
			\item if $g_k = \epsilon$ set $g_k = w$.
			\item otherwise, set $h_k = w$ and assign $k \leftarrow k+1$.
		\end{itemize}
		\item For the machine $T_g$ return $g_n$, for $T_h$ return $h_n$.
	\end{itemize}
	
	As $G$ is infinite and finitely generated there exist elements of arbitrary length. Therefore the bound $N$ suffices to construct all these disjoint balls: Indeed, it is the sum of the diameters of the considered sets. Moreover, as the lexicographic order is fixed beforehand this algorithm will always produce the same values, therefore it gives a recursive enumeration of the desired sets.\end{proof}

\begin{theorem}
\label{theorem.amenable_sym_non_sofic}
Let $G$ be an infinite amenable group. Then there exists a $G$-effectively closed subshift which is not sofic.
\end{theorem}

\begin{proof}
	Let $(g_n)_{n \in \NN}$, $(h_n)_{n \in \NN}$ be recursive sequences with oracle $\WP(G)$ as in Lemma~\ref{lemma_secuencias_recursivas_disjuntas}, and consider the subshift $Y \subset \{0,1,2\}^{G}$ defined as $Y = Y_1 \cap Y_2$ where:
	$$ Y_1 = \{ y \in \{0,1,2\}^{G} \mid 2\in \{y_g,y_h\} \implies g = h \}$$
	$$ Y_2 = \{ y \in \{0,1,2\}^{G} \mid y_g = 2 \implies \forall n \in \NN, \sigma_{g_n^{-1}g^{-1}}(y)|_{B_n} = \sigma_{h_n^{-1}g^{-1}}(y)|_{B_n} \}$$
	
It is clear these two sets are closed and shift-invariant, thus $Y$ is a subshift. Moreover, they are both $G$-effectively closed subshifts: $Y_1$ is defined by all pattern codings which contain a pair $(w_1,2), (w_2,2)$ such that $w_1 \neq_G w_2$ and $Y_2$ by all pattern codings which contain a triple $(w_1,2), (w_2,a),(w_3,b)$ with $a \neq b$ for which there exists $n \in \NN$ and $h \in B_n$ such that $w_2 =_G w_1g_nh$ and $w_3 =_G w_1h_nh$. As the sequences are recursive with $\WP(G)$ as oracle this is an effectively enumerable set with oracle $\WP(G)$. As the class of $G$-effectively closed subshifts is closed under intersections we obtain that $Y$ is $G$-effectively closed.

We are going to show that $Y$ is not sofic. As $G$ is amenable (see~\cite{ceccherini-SilbersteinC09}), for each $\varepsilon > 0$ and finite $K \subset G$ there exists a non-empty finite set $F \subset G$ such that:
$$\forall k \in K \text{,    }\frac{|F \setminus Fk|}{|F|} < \varepsilon$$

Suppose $Y$ is sofic, then there exists an SFT $X \subset \bg^G$ and a factor code $\phi: X \twoheadrightarrow Y$. Without loss of generality one can suppose that $\phi$ is a 1-block code, that is, it is defined by a local rule $\Phi: \bg \to \ag$. Indeed, if this was not the case, and $\Phi : \bg^F \to \ag$ for $F \neq \{1_G\}$ we can find a conjugated version of $X$ over the alphabet $\widetilde{\bg} := \bg^F$ which is given by the conjugacy $\widetilde{\phi}: X \to \widetilde{X}$ such that $\widetilde{\phi}(x)_g = \sigma_{g^{-1}}(x)|_F$. As being SFT is a conjugacy invariant we can choose without loss of generality $\widetilde{X}$ as the extension.

Let $K$ be the union of the supports of $p \in \FF$ where $X = X_{\FF}$ and $|\FF| < \infty$,  $\varepsilon = \frac{log(2)}{|K|log(|\bg|)}$ and for simplicity denote $\partial_K F = F \setminus \bigcap_{k \in K}Fk$. We obtain that there is $F$ such that: $$ \frac{|\partial_K F|}{|F|} \leq \sum_{k \in K} \frac{|F \setminus Fk|}{|F|} < |K|\frac{log(2)}{|K|log(|\bg|)} = \frac{log(2)}{log(|\bg|)}$$

Note that the previous property is invariant by translation, that is, if $F$ satisfies this property, then $gF$ also does for each $g \in G$. By choosing a large enough $n \in \NN$ such that $F \subset B_n$, then $g_nF \subset g_nB_n$.

Putting everything together, we can find a set $F$ such that $|\bg|^{|\partial F|} < 2^{|F|}$ and there exists $n \in \NN$ such that $1_G \notin g_nF$, $g_nF \subset g_nB_n$ and $g_nF \cap h_nB_n = \emptyset$.

Consider the set of patterns: $$\mathcal{P} = \{p : \{1_G\} \cup g_nF \to \{0,1,2\} \mid p_{1_G} = 2, \forall h \in g_nF : p_h \in \{0,1\} \}$$ Clearly $|\mathcal{P}| = 2^{|F|}$. As $g_nF \subset g_nB_n$ then for each $p \in \mathcal{P}$, $[p]_{1_G} \cap Y \neq \emptyset$. Let $y^p \in [p]_{1_G} \cap Y$ and $x^p \in X$ such that $\phi(x^p)=y^p$. As $|\bg|^{|\partial F|}<2^{|F|}$ by pigeonhole principle there are $x^{p_1} \neq x^{p_2}$ such that $x^{p_1}|_{g_n\partial F} = x^{p_2}|_{g_n\partial F}$.

By definition of $K$ we obtain that $\widetilde{x} \in X$ where $\widetilde{x}$ is the configuration defined as $\widetilde{x}|_{F}= x^{p_1}|_{F}$ and $\widetilde{x}|_{G \setminus F}= x^{p_2}|_{G \setminus F}$. As $\phi$ is a 1-block code we get that $\phi(\widetilde{x})|_{F}= y^{p_1}|_{F}$ and $\phi(\widetilde{x})|_{G \setminus F}= y^{p_2}|_{G \setminus F}$. Consider $\bar{g} \in B_n$ such that $(y^{p_1})_{g_n
\bar{g}} \neq (y^{p_2})_{g_n\bar{g}}$. Then:  

$$\phi(\widetilde{x})_{{h_n\bar{g}}} = (y^{p_2})_{h_n\bar{g}} = (y^{p_2})_{g_n\bar{g}}$$
$$\phi(\widetilde{x})_{{g_n\bar{g}}} = (y^{p_1})_{g_n\bar{g}}$$
Therefore $\phi(\widetilde{x})_{{h_n\bar{g}}} \neq \phi(\widetilde{x})_{{g_n\bar{g}}}$ but $\phi(\widetilde{x})_{1_G} = 2$ which means that $\phi(\widetilde{x}) \notin Y$. \end{proof}

In particular, this theorem gives a negative answer in the case of $BS(1,2)$ which is solvable and thus amenable.

\begin{definition}
The \emph{number of ends} $e(G)$ of the group $G$ is the limit as $n$ tends to infinity of the number of infinite connected components of $\Gamma(G,S) \setminus B_n$.
\end{definition}

The number of ends is a quasi-isomorphism invariant and thus it does not depend on the choice of $S$. It is also known that for a finitely generated group~$G$ then $e(G) \in \{0,1,2,\infty\}$. Stallings theorem about ends of groups \cite{STALLINGS1968} gives a constructive characterization of the groups satisfying $e(G) \geq 2$. In particular we have $e(G)=2$ if and only if $G$ is infinite and virtually cyclic.


\begin{theorem}\label{theorem.more_two_ends_stricly_sofic}
Let $G$ be a finitely generated group where $e(G) \geq 2$. Then there are $G$-effectively closed subshifts which are not sofic.
\end{theorem}

\begin{proof}

Let $N \in \NN$ such that $\Gamma(G,S) \setminus B_N$ contains at least two different infinite connected components $C_1$ and $C_2$.

Let $(g_i)_{i \in \NN} \subset C_1$ and $(h_i)_{i \in \NN} \subset C_2$ be sequences with no repeated elements. Let $Y\subset \{0,1,2\}^G$ defined as $Y = Y_1 \cap Y_2$ where:
$$ Y_1 = \{ y \in \{0,1,2\}^{G} \mid 2\in \{y_g,y_h\} \implies g = h \}$$
$$ Y_2 = \{ y \in \{0,1,2\}^{G} \mid y_g = 2 \implies \forall n \in \NN, y_{gg_n} = y_{gh_n} \}$$ 

Analogously to the proof of Theorem~\ref{theorem.amenable_sym_non_sofic}, if the sequences are recursive with oracle $\WP(G)$ then $Y$ is effectively closed. We claim such sequences exist.

Fix a total order on $S$ and extend it to a lexicographic order in $S^*$. Let $N$ as above and let $w_0 \in S^*$ such that $w  =_G g_0 \in C_1$. Consider the Turing machines $T_g$ with oracle $\WP(G)$ that  on entry $n \in \NN$:
\begin{itemize}
	\item If $n = 0$ returns $w_0$.
	\item Let $M = N+n+|w_0|$. Solve the word problem for every $w \in S^*$ such that $|w| \leq 2M$. This allows to construct $B_{M}$ of $\Gamma(G,S)$.
	\item Let $H_{g_0}$ be the connected component of $B_M \setminus B_N$ which contains $g_0$.
	\item Assign the value $0$ to every $H_{g_0} \setminus \{w_0\}$. and $1$ to $w_0$. Assign $g_1,\dots,g_n$ to $\epsilon$. And initiate a variable $k$ with its value set initially to $1$.
	\item While $k \leq n$ do the following: Iterate over all $w \in S^*$ lexicographically. If $w_0w$ has the value $0$ and belongs to $H_{g_0}$ then: 
	\begin{itemize}
		\item Turn the value $w_0w$ to $1$.
		\item Assign $g_k = w_0w$ and increase $k$ by $1$.
	\end{itemize}
	\item Return $g_n$.
\end{itemize}

As the component $C_1$ is infinite, the value $M$ suffices to find $n$ different elements. It is clear this machine yields a sequence of distinct elements in component $C_1$. The machine $T_h$ for the sequence in the component $C_2$ is analogous.

Suppose $Y$ is sofic. As in Theorem~\ref{theorem.amenable_sym_non_sofic} we can consider an SFT extension $X \subset \bg^G$ given by a 1-block code $\phi: X \twoheadrightarrow Y$. Let also $M \in \NN$ be a bound such that the union of all the supports of one finite set of forbidden patterns defining $X$ is contained in $B_M$. Let $L = N+M$.

As $G$ is finitely generated $|B_L| < \infty$. Consider thus the finite set $\mathcal{P} = \{ p \in \bg^{B_L} \mid \phi([p]_{1_G})\cap [2]_{1_G} \neq \emptyset \}$. Clearly $|\mathcal{P}| \leq |\bg|^{|B_L|} < \infty$. Consider $w \in \{0,1\}^{\NN}$ and fix $y^w \in \bigcap_{n \in \NN}[w_n]_{g_n} \cap [2]_{1_G}$. Clearly $y^w \in Y$. As there is an infinite number of such $y_w$ there exist $w_1 \neq w_2$ and $x^{w_1},x^{w_2} \in X$ such that $\phi(x^{w_1})= y^{w_1}$ and $\phi(x^{w_2})= y^{w_2}$ and $x^{w_1}|_{B_L} = x^{w_2}|_{B_L}$. 

By definition of $L$ we have that $\widetilde{x} \in X$ where:

$$\widetilde{x}_g =
\left\{
\begin{array}{l}
(x^{w_1})_g \mbox{, if } g \in C_1 \\ 
(x^{w_2})_g \mbox{, if } g \in G \setminus C_1
\end{array}
\right.
$$

Thus $\widetilde{y}=\phi(\widetilde{x})$ satisfies that $\widetilde{y}_{1_g}= 2$, $\widetilde{y}|_{C_1} = (y^{w_1})|_{C_1}$ and $\widetilde{y}|_{C_2} = (y^{w_2})|_{C_2}$. Let $n \in \NN$ such that $(w_1)_n \neq (w_2)_n$ Then: $\widetilde{y}_{g_n} = (y^{w_1})_{g_n}$ and  $\widetilde{y}_{h_n} = (y^{w_2})_{h_n} = (y^{w_2})_{g_n}$. Therefore $\widetilde{y} \notin Y_2$ which implies that $\widetilde{y} \notin Y$ .\end{proof}

\section{$G$-machines}
\label{section.G_machines}

 Classical Turing machines keep their information in a bi-infinite tape, and are only able to work on inputs which are codified in the form of words. While in~$\ZZ$ this a natural model to study subshifts, it becomes cumbersome in general groups as we are forced to introduce pattern codings. Moreover, as we saw in Section~\ref{section.effectiveness}, there is a number of constraints to what can be done with Turing machines when $\WP(G)$ is undecidable, and a general setting forces the use of oracles.

 In this section we introduce an alternative model of computation which we call a $G$-machine. In this model, the tape is replaced by a finitely generated group $G$. These machines receive patterns $p \in \ag_G^*$ as input instead of words and move by using the set $S$ of generators. Similar machines using Cayley graphs as a tape have already been mentioned in~\cite{Gajardo200734} and studied in more detail in~\cite{DaCunha2011}, but these machines take their input as a word in an auxiliary tape and only use the graph as a working tape.
 
 We begin this section by defining $G$-machines and the classes of languages they define. Then we present some robustness results similar to the ones satisfied by classical Turing machines. As the main result of this section, we characterize the class of $G$-effectively closed subshifts as those whose set of forbidden patterns is $G$-recursively enumerable, hence giving a characterization of this class without the use of oracles. We end this section with two applications of these machines: one to the domino problem of general groups (Theorem~\ref{theorem_group_undecidable_DP}) and another in the form of a simulation theorem (Theorem~\ref{Teorema_simulacion}).

\begin{definition}
A \define{$G$-machine} is a 6-tuple $(Q,\Sigma,\sqcup,q_0,Q_F,\delta)$ where $Q$ is a finite set of states, $\Sigma$ is a finite alphabet, $\sqcup \in \Sigma$ is the blank symbol, $q_0 \in Q$ is the initial state, $Q_F \subset Q$ is the set of accepting states and $\delta : \Sigma \times Q \to \Sigma \times Q \times S$ is the transition function.
\end{definition}

As in the case of Turing machines, we can define the action of a Turing machine in two different ways. We call these the fixed head and moving head models.

In the fixed head model, a $G$-machine $T$ acts on the set $\Sigma^G \times Q$ as follows: let $(x,q) \in \Sigma^G \times Q$ and $\delta(x_{1_G},q) = (a,q',s)$. Then $T(x,q) = (\sigma_{s^{-1}}(\widetilde{x}),q')$ where $\widetilde{x}|_{1_G} = a$ and $\widetilde{x}|_{G \setminus \{1_G\}} = x|_{G \setminus \{1_G\}}$. Figure~\ref{figure.simulation_g_machine} illustrates this action when $G$ is a free group. Here the head of the Turing machine is assumed to stay at a fixed position and the tape moves instead.

\begin{figure}[ht]
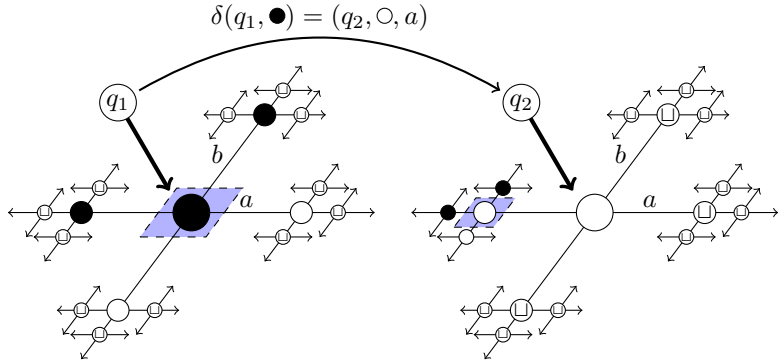

\begin{bigcenter}
\include{simulation_g_machine}
\caption{A fixed head transition of an $F_2$-machine.}
\label{figure.simulation_g_machine}
\end{bigcenter}
\end{figure}

In the moving tape model, a $G$-machine $T$ acts on the set $\Sigma^G \times G \times Q$ as follows: let $(x,g,q) \in \Sigma^G \times G \times Q$ and $\delta(x_{1_G},q) = (a,q',s)$. Then $T(x,g,q) = (\widetilde{x},gs,q')$ where $\widetilde{x}|_{1_G} = a$ and $\widetilde{x}|_{G \setminus \{1_G\}} = x|_{G \setminus \{1_G\}}$. Figure~\ref{figure.simulation_z2_machine} illustrates this action when $G$ is $\ZZ^2$. Here the tape remains fixed and the second coordinate keeps track of the position of the head.

\begin{figure}[ht]
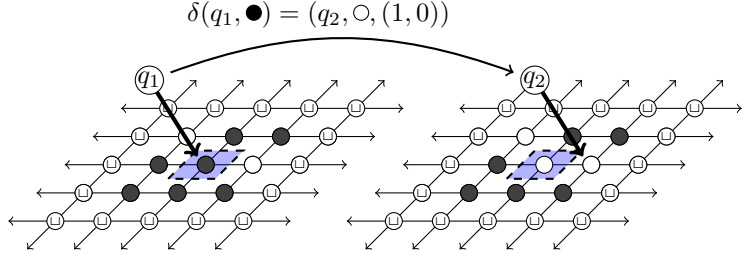

	\begin{bigcenter}
		\include{simulation_z2_machine}
		\caption{A moving head transition of a $\ZZ^2$-machine.}
		\label{figure.simulation_z2_machine}
	\end{bigcenter}
\end{figure}

Let $F \subset G$ be a finite set and $p \in \Sigma^F$. Let $x^p \in \Sigma^G$ be the configuration such that $(x^p)|_F = p$ and $(x^p)|_{G \setminus F} \equiv \sqcup$. We say that $T$ \define{accepts} $p$ if there is $n \in \NN$ such that $T^n(x^p,q_0) \in \Sigma^G \times Q_F$ in the fixed head model or equivalently $T^n(x^p,1_G,q_0) \in \Sigma^G \times G \times Q_F$ in the moving head model. $L \subset \Sigma_G^*$ is \define{$G$-recursively enumerable} if there exists a $G$-machine $T$ which accepts $p \in \Sigma_G^*$ if and only if $p \in L$. If both $L$ and $\Sigma_G^* \setminus L$ are $G$-recursively enumerable we say $L$ is \define{$G$-decidable}.

So far we have defined these machines using a fixed set of generators $S$. In the next proposition we show that the languages defined by such machines do not depend of this arbitrary choice.

\begin{proposition}
Let $S,S'$ be finite subsets of $G$ such that $\langle S \rangle = \langle S' \rangle = G$. Let $L \subset \ag_G^*$ be recursively enumerable using $S'$ as the movement set. Then $L$ is recursively enumerable using $S$.
\end{proposition}

\begin{proof}
Let $T_{S'}$ be a $G$-machine using $S'$ as the movement set recognizing $L$. As $\langle S \rangle = G$ each $s' \in S'$ can be written as $s' = s_1\dots s_{n(s')}$ where every $s_i \in S$. Consider $T_{S}$ a copy of $T_{S'}$ where for each state $q \in Q$ we add a copy $q_{s',s_i}$ for $s' \in S$ and $i \in \{1,\dots,n(s')\}$, and every instruction $\delta(a,q) = (b,r,s')$ in $T_{S'}$ is replaced with the instructions: 

\begin{itemize}
	\item $\delta(a,q)=(b,r_{s',s_1},s_1)$
	\item $\forall a \in \Sigma$ and $1 \leq i < n(s)$, $\delta(a,r_{s',s_i})= (a,r_{s',s_{i+1}},s_{i+1})$
	\item $\forall a \in \Sigma$, $\delta(a,r_{s',s_{n(s')}}) = (a,r,1_G)$.
\end{itemize}

The modified machine $T_{S}$ moves with the set of generators $S$ and acceps the same patterns as $T_{S'}$.\end{proof}

The class of $G$-machines shares also the robustness of Turing machines with respect to slight changes in its definition. For example, we can allow multiple tapes with multiple independent writing heads. We shall briefly and informally define this model as it will be used as a tool in a proof later on.

A \define{multiple head $G$-machine} is the same as a $G$-machine, except that the machine uses $G^n$ as a tape and the transition function is $\delta :  \Sigma^n \times Q^n \to \Sigma^n \times Q^n \times S^n$, where $n$ is the number of heads of the machine. The action of this machines is defined analogously as before in either the moving head or moving tape model. It \define{accepts} a pattern $p \in\ag_G^*$ if starting from the initial configuration $( (x^p,\sqcup^G,\dots, \sqcup^G), (q_0,\dots, q_0))$ the machine reaches in a finite number of steps a configuration with an accepting state in $Q_F$ in one of the coordinates.

In these machines each head works on its own tape, but can ``read'' the content of other tapes. By codifying independent movements of a tape accordingly, it is able to read not only what each head is looking at a certain step but what is written in an arbitrary finite portion of the other tapes. 

\begin{proposition}
\label{proposition.multiple_heads_Gmachine}
Let $L \subset \Sigma_G^*$. There exists a multiple head $G$-machine which accepts exactly patterns $p \in L$ if and only if $L$ is $G$-recursively enumerable.
\end{proposition}

This extended model is useful to prove the second of the following two results which link oracle machines to $G$-machines. The first result is relatively straightforward, as $G$-machines can be simulated by a machine with oracle $\WP(G)$ by creating arbitrarily big balls of the Cayley graph. The second result is more interesting as it says that oracle machines can be simulated by $G$-machines.

\begin{theorem}\label{theorem_g_effective_is_oracle}
Let $L \subset \Sigma_G^*$ be $G$-recursively enumerable. Then there exists a recursively enumerable with oracle $\WP(G)$ set of pattern codings $\CC$ such that $L = p(\CC)$.
\end{theorem}

\begin{proof}

Suppose $T_G$ is the $G$-machine recognizing $L$. With an oracle of $\WP(G)$, a machine can construct balls $B_n$ of $\Gamma(G,S)$ for arbitrary $n$. A codification of $B_n$ allows a classical Turing machine to simulate at least $n$ applications of $T_G$ in the moving head model as the head starts in the origin and moves at most one generator per iteration.
Let $T$ be the Turing machine with oracle $\WP(G)$ which does the following on entry $c$.

\begin{itemize}
	\item Let $N = 2\max_{(w,a)\in c}{|w|}$. Solve the word problem for all $w \in S^*$ of length at most $N$. If $c$ is inconsistent accept. 
	\item Let $k = N$ and iterate the following procedure: Solve the word problem for $w \in S^*$ of length at most $k$ and simulate $T_G$ over $p(c)$ for $k$ steps. If this procedure accepts then accept, otherwise increase $k$ by $1$.
\end{itemize}

Clearly, $T$ accepts $c$ if and only if either $c$ is inconsistent or $p(c) \in L$. \end{proof}

\begin{definition}
	A language $L \subset \Sigma_G^*$ is said to be \define{closed by extensions} if for each $p_1 \in \Sigma^{F_1}$, $p_2 \in \Sigma^{F_2}$ such that $F_1 \subset F_2$ and $p_2|_{F_1} = p_1$ then $p_1 \in L \implies p_2 \in L$.
\end{definition}

\begin{theorem}\label{theorem_oracle_is_G_effective}
Let $G$ be an infinite group and $\CC$ a recursively enumerable with oracle $\WP(G)$ set of pattern codings. If $p(\CC)$ is closed by extensions, then $p(\CC)$ is $G$-recursively enumerable.
\end{theorem}

\begin{proof}
Without loss of generality we can suppose $\CC$ is a maximal set of pattern codings which gives $p(\CC)$. Moreover we can also assume that $T$ is a one-sided Turing machine with a reading tape and a working tape.

The construction is a multiple head $G$-machine $\mathcal{M}$ which consists of the following six layers (see Figure~\ref{figure.Z_machine_inside_G_bis}):

\begin{enumerate}
\item A storage layer $\mathcal{M}_{\text{STORE}}$ where the input $p \in \Sigma_G^*$ is stored.
\item A machine $\mathcal{M}_{\text{PATH}}$ which constructs an arbitrarily long one-sided non-intersecting path starting from $1_G$.
\item A machine $\mathcal{M}_{\text{VISIT}}$ which is able to visit iteratively all the elements of $B_n$ for $n \in \NN$ starting with $n$ initially assigned to 1.
\item A Machine $\mathcal{M}_{\text{ORACLE}}$ which solves $\WP(G)$.
\item An auxiliary layer $\mathcal{M}_{\text{AUX}}$ which serves as a nexus between the first layer and the sixth.
\item A simulation layer $\mathcal{M}_{\text{SIM}}$ which simulates $T$ in the one-sided path created by $\mathcal{M}_{\text{PATH}}$.
\end{enumerate}

\begin{figure}[h]
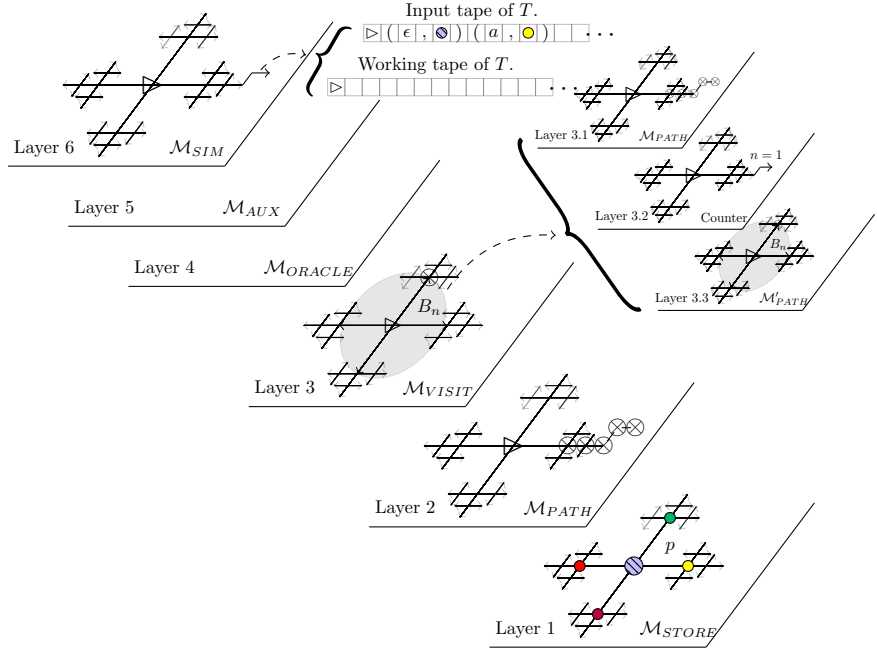

	\begin{center}
		\include{Z_machine_inside_G_bis}
		\caption{Construction of the machine $\mathcal{M}$ as a multiple head $G$-machine.}
		\label{figure.Z_machine_inside_G_bis} 
	\end{center}
\end{figure}

We will first describe $\mathcal{M}_{\text{PATH}}$ and $\mathcal{M}_{\text{VISIT}}$ which are the most complicated components. Then we will describe the general working of the machine.

We begin by describing $\mathcal{M}_{\text{PATH}}$ in detail. Let the set of generators $S= \{g_1,g_2,\dots,g_k\}$ and consider the $G$-machine $\mathcal{M}_{\text{PATH}} := (Q,\Sigma,\sqcup,q_0,Q_F\delta)$ where $Q := \{I,B\}\cup(S \times\{\leftarrow,\rightarrow\})$, $\Sigma = (\{\sqcup,\vartriangleright \}\cup S)\times\{\sqcup,\otimes\}\times(\{\sqcup\}\cup S)$, $q_0 = I$, $Q_F = \emptyset$ (we force the machine to loop), and $\delta$ is given by the following rules where $*_i$ stands for an arbitrary fixed symbol.

$$\delta((\sqcup,\sqcup,\sqcup),I) = ((\vartriangleright,\otimes,g_1), g_1^{\leftarrow} , g_1).$$ 
$$\delta((\sqcup,\sqcup,\sqcup),g_i^{\leftarrow}) = ((g_i,\otimes,\sqcup), g_1^{\rightarrow} , 1_G).$$
$$\delta((*_1,\otimes,*_2),g_i^{\rightarrow}) = ((*_1,\otimes,g_i), g_i^{\leftarrow} , g_i).$$
$$\delta((*_1,\otimes,*_2),g_i^{\leftarrow}) = 
((*_1,\otimes,*_2), B , g_i^{-1}).$$
$$\delta((g_j,\otimes,g_i),B) = 
\begin{cases} ((g_j,\otimes,g_i),g_{i+1}^{\rightarrow},1_G), & \mbox{if }  i < k \\ ((\sqcup,\sqcup,\sqcup),B,g_j^{-1}),  & \mbox{if } i = k. \end{cases}$$
$$\delta((\vartriangleright,\otimes,g_i),B) =  ((\vartriangleright,\otimes,g_i),g_{i+1}^{\rightarrow},1_G), \mbox{   if }  i < k $$

The rules from $\delta$ codify a backtracking in $G$ which marks a one-sided non-intersecting infinite path in $G$. The states $I$ and $B$ stand for initialization and backtracking respectively. The elements from $\Sigma$ are triples $(a_1,a_2,a_3)$ which indicate the following information: my left and right neighbors are $a_1$ and $a_3$ respectively and I belong to the path if $a_2 = \otimes$. The first rule initializes the infinite path by using the symbol $\vartriangleright$ to indicate that there is no element to the left, marks the identity of the group as part of the path by using $\otimes$ and sets the next element in the direction $g_1$. The second and third rules mark the left and right neighbors respectively and move to the next position. Rule 4 deals with the case of reaching a position already marked and going back. Rule 5 and 6 search the next available direction which potentially admits an infinite path and backtrack if every position has already been searched. Rule 6 lacks a case where $i = 
2k$ on purpose because such a state is never reached as the group is infinite.

Next we describe $\mathcal{M}_{\text{VISIT}}$ that visits all elements of every ball $B_n$ in $G$ iteratively. It suffices to construct it as a multiple head $G$-machine with three layers as follows. The first layer runs a copy of $\mathcal{M}_{\text{PATH}}$. The second layer makes use of the path defined by $\mathcal{M}_{\text{PATH}}$ to simulate a counter which has value $n \in \NN$ -- any one-sided Turing machine can be simulated in the path by identifying the instructions $L,R$ with the first and third coordinates of $\Sigma$. The third layer runs another copy of $\mathcal{M}_{\text{PATH}}$, which is allowed only to run over words of length $n$. This is achieved by using the counter in second layer to measure the length of the path visited by the third layer and restrict it to be less than $n$. Each time the whole ball $B_n$ is visited (that is, $((\vartriangleright,\otimes,g_{k}),B)$ is reached in the third layer) then the counter in the second layer increments $n$ by $1$ and the third layer starts anew.
 
If at a given time the first layer, which constructs the one-sided path, backtracks until reaching a cell used by the counter in the second layer, then the second and third layers are erased and restart. As the group is infinite, then by choosing an adequate number of computation steps, the path generated by $\mathcal{M}_{\text{PATH}}$ in the first layer is arbitrarily long. Thus the head of the third tape is able to visit every element of $B_n$ for arbitrarily big $n$.

\medskip

Finally, we describe the overall functioning of $\mathcal{M}$:

\begin{itemize}
	\item The input $p \in \Sigma_G^*$ is stored in $\mathcal{M}_{\text{STORE}}$ whose head mimics that of $\mathcal{M}_{\text{VISIT}}$ without changing anything.
	\item The machines $\mathcal{M}_{\text{PATH}}$ and $\mathcal{M}_{\text{VISIT}}$ run independently.
	\item $\mathcal{M}_{\text{SIM}}$ uses the path given by $\mathcal{M}_{\text{PATH}}$ to simulate two one-sided Turing machine tapes: an input tape where input will be stored, and a working tape which simulates $T$ over that input.
	\item Whenever $\mathcal{M}_{\text{VISIT}}$ arrives at a position where the first layer is not marked by $\sqcup$, the head at $\mathcal{M}_{\text{AUX}}$ follows the path $w$ marked from $1_G$ by the first layer of $\mathcal{M}_{\text{VISIT}}$ and writes $(w,a)$ in the input tape of $\mathcal{M}_{\text{SIM}}$. Then $\mathcal{M}_{\text{AUX}}$ marks position $w$ as already visited and returns to $1_G$.
	\item If at a given time $\mathcal{M}_{\text{AUX}}$ extends the pattern coding written in the reading tape of the fifth layer, then the working tape of $\mathcal{M}_{\text{SIM}}$ erases everything and begins anew.
	\item If at any moment the working tape of $\mathcal{M}_{\text{SIM}}$ makes a call to the oracle~$\WP(G)$, then $\mathcal{M}_{\text{ORACLE}}$ is made to mark the origin, follow the path $w \in S^*$ and accept the call if the last symbol is marked. Then it erases everything and goes back to the origin.
	\item If at any moment the end of the simulated path 
	created by $\mathcal{M}_{\text{PATH}}$ backtracks into a cell used by the written portion of $\mathcal{M}_{\text{SIM}}$, then the content of all tapes except $\mathcal{M}_{\text{PATH}}$ and $\mathcal{M}_{\text{STORE}}$ is erased and they start anew. 
	\item $\mathcal{M}$ accepts if and only if the working tape of $\mathcal{M}_{\text{SIM}}$ does.
\end{itemize}

As $\mathcal{M}_{\text{PATH}}$ is able to construct arbitrarily long one-sided and non-intersecting paths, there is a finite number of computation steps such that $\mathcal{M}_{\text{VISIT}}$ will visit all of the support of $p$. Thus the fourth layer will write a consistent pattern coding $c$ such that $p = p(c)$ which is accepted by the working tape of $\mathcal{M}_{\text{SIM}}$ if and only if $p \in p(\CC)$ (as $\CC$ is maximal). By considering a path which has length at least two times the running time of all the other algorithms, this eventually happens. Conversely, if $p \notin p(\CC)$, as $p(\CC)$ is closed by extensions, the acceptance of any partial coding $c'$ would mean that $p \in p(\CC)$, therefore, the machine never accepts. \end{proof}

\begin{corollary}\label{the_great_corollary}
	A subshift $X \subset \ag^G$ is $G$-effectively closed if and only if there exists a $G$-recursively enumerable set $\FF \subset \ag_G^*$ such that $X = X_{\FF}$.
\end{corollary}

\begin{proof}
	As $X$ is $G$-effectively closed, the set of forbidden pattern codings $\CC$ can be chosen to be maximal. This in turn gives a maximal set of forbidden patterns $p(\CC)$ which is closed by extensions. Theorems~\ref{theorem_g_effective_is_oracle} and~\ref{theorem_oracle_is_G_effective} imply the result.
\end{proof}

Let $\texttt{HALT}_{G} = \{ \langle T \rangle \mid T \text{  is a }G\text{-machine which accepts the empty input}\}$. 

\begin{corollary}\label{hardness}
	Let $G$ be an infinite group. $\texttt{HALT}_{G}$ is $\WP(G)'$-hard, that is, it is at least as hard as the halting problem for Turing machines with oracle $\WP(G)$
\end{corollary}

\begin{proof}
	Let $T$ be a Turing machine with oracle $\WP(G)$. Consider the construction from Theorem~\ref{theorem_oracle_is_G_effective} without the Visit and Auxiliary tapes. Thus, there is only the tape which searches the infinite path, the oracle layer, and the layer which simulates $T$ (now only on empty input). It is clear that this machine accepts the empty input (and all inputs) if and only if $T$ accepts the empty input.
\end{proof}

Corollary~\ref{the_great_corollary} implies that $G$-effectively closed subshifts can be defined either by oracle machines or by $G$-machines. This nice characterization allows us to simulate Turing machines in groups which may not even have torsion-free elements. In what remains of this section we present applications of these machines both to the domino problem and to construct a simulation theorem.


\subsection{Application: A class of groups with undecidable domino problem}
\label{subsection.DP_undecidable}

The \define{domino problem} of a finitely generated group $G$ is defined as the language given by the finite sets of pattern codings which give an empty subshift. Informally:

$$\texttt{DP}(G) = \{ \langle\FF\rangle \mid |\FF| < \infty, X_{\FF} = \emptyset \}$$,

where $\langle\FF\rangle$ is a codification of the finite set of forbidden patterns $\FF$. Another related notion is the \define{origin constrained domino problem}, where a symbol in the alphabet is fixed to appear at the origin.

$$\texttt{OCDP}(G) = \{ (a,\langle\FF\rangle) \mid |\FF| < \infty, X_{\FF}\cap [a]_{1_G} = \emptyset \}$$

%

It is known that both problems are decidable in $\ZZ$~\cite{lind1995introduction} but undecidable in $\ZZ^d$ with $d >1$~\cite{Berger1966,Robinson1971}. Clearly, the decidability of $\texttt{OCDP}(G)$ implies the decidability of $\texttt{DP}(G)$ as it would suffice to run the algorithm for every symbol of the finite alphabet. So far, we do not know any group where the decidability of these two languages differ. In this section we use $G$-machines to exhibit a class of groups where these problems are undecidable.

\begin{theorem}
	\label{theorem_group_undecidable_DP} 
	Let $G$ be an infinite group with the special symbol property. Then:
	\begin{itemize}
		\item The origin constrained domino problem $\texttt{OCDP}(G \times \ZZ)$ is $\texttt{WP}(G)'$-hard. 
		\item For any non-trivial finite group $H$, the domino problem for $(G \times \ZZ) \ast H$ is $\texttt{WP}(G)'$-hard. 
		
	\end{itemize}

\end{theorem}
\begin{proof}
	
	Let $G = \langle S \rangle$ and $T$ a $G$-machine with tape alphabet $\Sigma = \{\sqcup,0,1\}$ and transition function $\delta: \Sigma \times Q \to \Sigma \times Q \times S$. Denote the states by $Q = \{1,\dots,k\}$ where the initial and final states are $1$ and $k$ respectively. Finally, let $\ag = \Sigma \times \{0,\dots,k\}$ and $Z = \Sigma^G \times X_{\leq k} \subset \ag^G$ where: $$X_{\leq k} = \{x \in \{0,\dots,k\}^G \mid  0 \notin \{x_g,x_h\} \implies g = h \}.$$
	
	The subshift $Z$ consists on configurations where there is at most one appearance of a state in $Q$. As $G$ satisfies the special symbol property, this is a sofic subshift. 
	
	We are going to define an extended sofic subshift $Y \subset \ag^{G \times \ZZ}$ which simulates the dynamical behavior of $T$. We do this by defining its set of forbidden patterns as $\FF = A_1 \cup A_2 \cup A_3 \cup A_4$ where these four sets are defined as follows: 
	
	\begin{itemize}
		\item Let $F \subset G$. For $p \in \ag^F$ we define its immersion $\gamma(p): \ag^{F} \to \ag^{F \times \{0\}}$ where $\gamma(p)_{(g,0)} = p_g$ for every $g \in F$. We define $A_1$ as the immersion of the forbidden patterns defining $Z$.
		\item Consider the support $F = \{(1_G,0), (1_G,1)\}$. We define $A_2$ as the set of $p \in \ag^F$ such that $p_{(1_G,0)} = (a,0)$ and $p_{(1_G,1)} = (b,\cdot)$ with $b \neq a$.
		\item Let $\delta(a,q)=(b,r,s)$. We define $A_3 = B_1 \cup B_2$ where these sets are the following:
		\begin{itemize}
			\item Let $F = \{(1_G,0), (1_G,1)\}$, we define $B_1$ as the set of $p \in \ag^F$ such that $p_{(1_G,0)} = (a,q)$ and $p_{(1_G,1)} = (c,\cdot)$ with $c \neq b$.
			\item Let $F_s = \{(1_G,0), (s,1)\}$, we define $B_2$ as the set of $p \in \ag^{F_s}$ such that $p_{(1_G,0)} = (a,q)$ and $p_{(s,1)} = (\cdot,t)$ with $t \neq r$.
		\end{itemize}
		\item We define $A_4$ to be the patterns with support $\{(1_G,0)\}$ containing the symbol $(a,k)$ for some $a \in \Sigma$.
	\end{itemize}
	
	This subshift is clearly sofic, as its forbidden patterns are the immersion of the forbidden patterns of a sofic subshift plus a finite amount of new forbidden patterns. The set $A_1$ just forces every coset $(G,z)$ to contain a configuration of~$Z$. Said otherwise, at most one head. $A_2$ forces that whenever a state $0$ appears then the symbol must remain unchanged. $A_3$ is composed of two rules related to the head: the first, $B_1$ forces the symbol in the head position to correspond to the one from the rule $\delta$. $B_2$ forces the movement of the head to correspond to the rule $\delta$. Finally, $A_4$ forbids the appearance of the final state $k$. 
	
	Consider the coding $\rho : \Sigma^G \times G \times Q \to Z$ given by $\rho(x,h,q) = (x,z)$, where $z_{h}=q$ and $z|_{G \setminus \{h\}} \equiv 0$. This coding takes a configuration in the moving tape model and represents it as an element of the subshift $Z$. The rules defining $Y$ force that if $y \in Y$ and $y_{(g,n)} =  \rho(x,h,q)_g$ for all $g \in G$, then for all $m \geq 0$ one has $y_{(g,n+m)} = \rho( T^m(x,h,q) )_g$.
	
	Using this previous relation and the fact that appearances of the final state are forbidden, we obtain that there exists $y \in Y$ such that $\forall g \in G$ then $y_{(g,0)} = \rho(\sqcup^G,h,1)$ for some $h\in G$ if and only if $T$ does not accept the empty input. 
	
	Let $X_{\text{aux}} \subset \{0,\star\}^{G \times \ZZ}$ be the SFT defined by the following forbidden patterns: for every $s\in S$ the pattern $p \in \{0,\ast\}^{\{1_H,s\}}$ such that $p_{1_H} = \star$ but  $p_s \neq \star$ is forbidden. This basically means that if a $\star$ appears in a position, then the whole $G$-coset contains a $\star$.
	
%

Now we have all the elements for the final construction: let $X_{\text{final}} \subset \ag^{G \times \ZZ} \times X_{\text{aux}}$ defined by the following forbidden patterns:
	
	
	\begin{itemize}
		\item The immersion of all forbidden patterns in $Y$.
	\item A $\star$ in $X_{\text{aux}}$ must always be accompanied by $(\sqcup,j)$ for some $j \in \{0,\dots,k\}$.
	\end{itemize}
	
This subshift is again sofic, since its forbidden patterns are the immersion of those of $Y$ and a finite number of forbidden symbols in the alphabet. The role of $X_{\text{aux}}$ is to force a $G$-coset to represent the machine $T$ starting on empty input. We claim that $X_{\text{final}} \cap [((\sqcup,1),\star )]_{1_G} = \emptyset$ if and only if $T$ accepts the empty input.

Indeed, let $x \in X_{\text{final}} \cap [((\sqcup,1),\star )]_{1_G}$. By using the definition of $X_{\text{aux}}$, the local rule of $X_{\text{final}}$ and the characterization of $Y$, we deduce that for all $g \in G \setminus \{1_G\}$ then $x_{(g,0)} = ( (\sqcup,0),\star)$. Therefore the projection $\pi_1$ to the first coordinate of $x$ would satisfy $\pi_1(x)|_{(g,0)} = \rho(\sqcup^G,h,1)_g$. This implies that $\pi_1(x)|_{h(g,k)}= \rho(T^m(\sqcup^G,h,1))_g$ for all $m \geq 0$. As the final state $k$ can not appear, we conclude that $T$ does not accept the empty input. Conversely, if $T$ does not accept the empty input we can construct a valid point as follows: let $y \in Y$ such that $\forall n \geq 0$ $y_{(g,n)} = \rho(T^n(\sqcup^G,1_G,1))$ and $\forall m \leq -1$ $y_{(g,m)} =(\sqcup,0)$. This is a valid point of $Y$ as $T$ does not accept the empty input. We can therefore define $x \in X_{\text{final}}$ as follows:

$$x_{(g,k)} =
\left\{
\begin{array}{l}
(y_{(g,0)},\star) \mbox{, if } k = 0 \\ 
(y_{(g,k)},0) \mbox{, if } k \neq 0\\ 
\end{array}
\right.
$$


%

which satisfies $x \in X_{\text{final}} \cap [((\sqcup,1),\star )]_{1_G}$. 

Note that the previous argument implies that if $((\sqcup,1),\star )$ appears in a configuration, it can only do so in at most one position. Therefore, we can consider a 1-block SFT extension of $X_{\text{final}}$ with at most one preimage $a$ of $((\sqcup,1),\star )$. This is a $G \times \ZZ$ SFT such that $X \cap [a] = \emptyset$ if and only if $T$ accepts the empty input. Therefore $\texttt{OCDP}(G \times \ZZ)$ is at least as hard as the halting problem for $G$-machines which in turn is $\WP(G)'$-hard by Corollary~\ref{hardness}.

Let $H$ be a finite group and consider the subshift $Y_{\text{aux}} \subset \{0,\ast\}^{(G \times \ZZ) \ast H}$ defined by the following forbidden patterns: $p \in \{0,\ast\}^{H}$ such that $|\{h \in H \mid p_h = \ast\}| \neq 1$. This means that every coset of $H$ must contain exactly one appearance of $\ast$. In the following, we choose a configuration $y \in Y_{\text{aux}}$ every $(G \times \ZZ)$-coset contains at most one occurrence of $\ast$. Let $\bar{h} \in H \setminus \{1_H\}$ and $w$ a reduced word representation of an element in $(G \times \ZZ) \ast H$. Define

$$y_{w} =
\left\{
\begin{array}{l}
1 \mbox{, if $w$ ends by } \bar{h} \\ 
0\mbox{, otherwise.} \\ 
\end{array}
\right.
$$

By using $Y_{\text{aux}}$ as an extra SFT layer, we can force the appearance of $((\sqcup,1),\star)$ every time an $\ast$ appears, and immerse the patterns of $X_{\text{final}}$ into $(G \times \ZZ) \ast H$. By definition, each configuration in $Y_{\text{aux}}$ has at least one coordinate marked by an $\ast$, and $y$ has also the property that each $(G \times \ZZ)$-coset contains at most one occurrence of $\ast$. We can thus repeat the previous argument to conclude that $\texttt{DP}( (G \times \ZZ) \ast H)$ is $\WP(G)'$-hard.\end{proof}

The role of the free product with $H$ is to ensure that the machine starts the calculation over an empty tape at some place in every configuration. In the classical construction of Robinson~\cite{Robinson1971} in the plane, this property is obtained using a hierarchical construction. We do not know if a generalization of this construction can be done in general groups.

Notice also that this result does not give new groups with undecidable domino problem when $G$ has at least one non-torsion element. Indeed, if $\ZZ$ embeds into $G$ then $\ZZ^2$ also embeds into $G \times \ZZ$. The advantage of this method using $G$-machines is that it allows to give a result over torsion groups such as the Grigorchuk group~\cite{Grigorchukgrouporiginal1984},

\subsection{Application: A simulation theorem with oracles}
\label{subsection.simulation}

In~\cite{AubrunSablik2010,DBLP:conf/birthday/DurandRS10} it is shown that every effectively closed subshift on $\ZZ$ can be obtained as the projective subdynamics of a sofic subshift on $\ZZ^2$. As Propositions ~\ref{proposition_subdyn_effective} and~\ref{proposition_subdyn_not_stable} show, an analogue can not hold for arbitrary $G$-effectively closed subshifts when $G$ is recursively presented, as the projective subaction would necessarily be effectively closed. Nevertheless, using $G$-machines, we can obtain a similar result if we allow the addition of a particular subshift as an universal oracle to our construction. Formally we show:

\begin{theorem}\label{Teorema_simulacion} For every finitely generated group $G$, there exists a $G \times \ZZ$-effectively closed subshift $U\subset \widetilde{\bg}^{G\times\ZZ}$ such that for every $G$-effectively closed subshift $X \subset \ag^G$ which contains a uniform configuration ($\exists \bar{a} \in \ag$ such that $\bar{a}^G \in X$), there exist an alphabet $\bg$, a finite set of forbidden patterns $\FF$ on alphabet $\widetilde{\bg} \times \bg$ and a 1-block code $\phi$ such that:

$$\pi_G \left(\phi \left( \left( U \times \bg^{G\times \ZZ}\right) \setminus \bigcup_{p \in \FF, h \in G \times \ZZ}[p]_h\right) \right) = X.$$
\end{theorem}

In order to define $U$ we need to introduce some technical constructions. Let $(X,d)$ be a metric space and $D \subset X$. The \emph{packing radius} of $D$ is  $r_D = \frac{1}{2} \inf{\{ d(x,y) \mid x,y \in D, x \neq y} \}$ and the \emph{covering radius} of $D$ is given by $c_D = \sup{\{ d(x,D) \mid x \in X \}}.$ Notice that for each pair of different $x,y \in D$, we have $B(x,r_D)\cap B(y,r_D) = \emptyset$ and $\bigcup_{x \in D}B(x,c_D) = X$. A set with non-zero packing radius and finite covering radius is said to be Delone. Notice that by definition a Delone subset of a non-empty set must be non-empty.

We define $Y_n \subset \{0,1,2\}^G$ as the subshift defined by the following set of forbidden patterns $\FF_n$:

\begin{itemize}
	\item All $p \in \{0,2\}^{B(1_G,4n)}$.
	\item $p \in \{0,1,2\}^{B(1_G,n)}$ such that $p_{1_G}=1$ and there exists $g \in B(1_G,n)\setminus \{1_G\}$ with $p_g \neq 2$.
	\item $p \in \{1,2\}^F$ where $F$ is a connected component of $\Gamma(G,S)$ and there exist $g_1,g_2 \in F, g_1 \neq g_2$ such that $p_{g_1}=p_{g_2}=1$.
\end{itemize}

That is, $Y_n$ is the set of configurations $y$ where, if we denote the set of positions marked in $y$ by a $1$ by $D_y$, then $D_y$ forms a Delone set with $r_{D_y} \geq n$ and $c_{D_y} \leq 4n$. Also, each $1$ is surrounded by a ball of size at least $n$ marked by $2$'s and there is no path of $2$'s connecting two adjacent $1s$. See Figure~\ref{figure.example_yn} for an example in $\ZZ^2$.

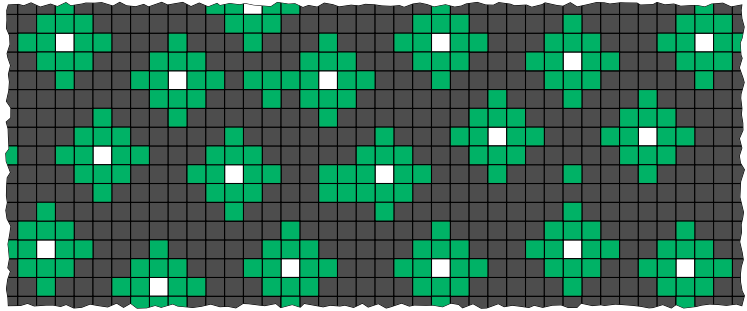
\begin{figure}[h!]
	\centering
	\include{example_yn}
	\caption{Example of a configuration of $Y_2$ for the group $\ZZ^2$ with the canonical generators. The symbols $0,1$ and $2$ are represented by the colors \protect\begin{tikzpicture}
		\protect\draw[fill = black!70] (0,0) rectangle ++(0.3,0.3);
		\protect\end{tikzpicture}, \protect\begin{tikzpicture}
		\protect\draw[fill = white] (0,0) rectangle ++(0.3,0.3);
		\protect\end{tikzpicture} and \protect\begin{tikzpicture}
		\protect\draw[fill = vert] (0,0) rectangle ++(0.3,0.3);
		\protect\end{tikzpicture} respectively.    }
	\label{figure.example_yn}
\end{figure}

%


\begin{claim*}
 $\forall n \geq 1$, $Y_n$ is a non-empty, $G$-effectively closed subshift.
\end{claim*}

\begin{proof}
	The set $\FF_n$ can easily be recognized by a Turing machine with oracle $\WP(G)$, so $Y_n$ is $G$-effectively closed. For the non-empty part, we claim a Delone set $D$ satisfying $r_{D} \geq 2n$ and $c_{D} \leq 4n$ always exists. Indeed, Consider the restriction to $B(1_G,k)$ for some $k \in \NN$ and choose a maximal set $D_k \subset B(1_G,k)$ with $r_{D_k} \geq 2$. If $c_{D_k} > 4n$ then the set $K := \{g \in B(1_G,k) \mid d(D_k,g)>2n\}$ is not empty and $D_k$ can be extended by an element of $K$, contradicting its maximality. Thus $c_{D_k} \leq 4n$. Now, consider the sequence of indicator functions of $(D_k)_{k \in \NN}$ and choose an accumulation point. This limit is the indicator function of a Delone set $D$ which satisfies the aforementioned property. Now, define $y \in \{0,1,2\}^G$ as:
	$$y_g = \begin{cases}
	1 & \text{ if } g \in D \\
	2 &  \text{ if } 0< d(g,D) \leq n \\
	0 &  \text{ else }
	\end{cases}$$
	As $c_D \geq 2n$ and $n \geq 1$ it follows that there is no path consisting of $2$'s between a pair of $1$'s. It follows that $y \in Y_n$. \end{proof}

Consider a $G$-machine $T$ with alphabet $\Sigma$ and set of states $Q$ whose head never leaves a bounded support $F$. Using a pigeonhole argument, it can be shown that if it accepts, it must do so before $|Q|\cdot|F|\cdot|\Sigma|^{|F|}$ steps. Consider the function $\texttt{time} : \NN \to \NN$ given by $\texttt{time}(n) = n^{n^n+n+1}$. It is clearly a computable function which satisfies the following property: for any $G$-machine $T$, there exists $N \in \NN$ such that for every $n \geq N$, if $T$ accepts a pattern $p$ without leaving the support $B(1_G,n)$ then it does so before $\texttt{time}(n)$ steps. Indeed, we can always bound $B(1_G,n) \leq |S|^n$ and thus an upper bound for the maximum number of steps without leaving the support $B(1_G,n)$ is given by $|Q|\cdot|S|^n\cdot|\Sigma|^{|S|^n}$. Choosing $N \geq \max\{|Q|,|S|,|\Sigma|\}$ we get that $\forall n \geq N$ the number of steps is bounded by $n^{n^n+n+1}$.

We are going to construct a $\ZZ$-subshift $X_{\texttt{time}}$ which encodes the function $\texttt{time}$ and instructions for a Turing machine in a convenient way. Consider the alphabet $\ag_X = \{ \bullet , \star, \oplus, \vartriangleright \} \cup S$. Let $\widetilde{x} \in \ag_X^{\NN}$ be the infinite concatenation of $\{w_n\}_{n \in \NN}$, where $w_0 = \star$ and for $n \geq 1$ the word $w_n$ is defined as follows. Let $u_1,\dots, u_{k(n)}$ be the lexicographic enumeration of all words in $S^*$ of length at most $4n$. Then, 
$$ v_{j,n} = u_{j}\vartriangleright\bullet^{\texttt{time}(n)}u_{j}^{-1}, \text{ and } w_n = \oplus v_{0,n}v_{1,n},\dots,v_{k(n),n}$$

\begin{example}
	Let $S = \{a,a^{-1}\}$ and suppose just for this example that the words are enumerated up to length $n$ instead of $4n$, and that $\texttt{time}(1)=2$ and $\texttt{time}(2) = 3$. Then the first symbols of $\widetilde{x}$ would be:
	\begin{align*}
	\widetilde{x} = & \star \oplus \vartriangleright \bullet \bullet  a \vartriangleright\bullet\bullet a^{-1} a^{-1}\vartriangleright \bullet\bullet a \oplus \vartriangleright\bullet\bullet\bullet  a \vartriangleright \bullet\bullet\bullet a^{-1} a^{-1} \vartriangleright\bullet\bullet\bullet  a  \\ & aa\vartriangleright\bullet\bullet\bullet a^{-1}a^{-1}  aa^{-1}\vartriangleright \bullet\bullet\bullet aa^{-1}a^{-1}a \vartriangleright \bullet\bullet\bullet a^{-1}a a^{-1}a^{-1}\vartriangleright\bullet\bullet\bullet aa \cdots
	\end{align*}
	
\end{example}

With the infinite word $\widetilde{x}$ in hand, we define $X_{\texttt{time}} \subset \ag_X^{\ZZ}$ as the subshift such that if $x \in X$ and $x_n = \star$, then for all $m \geq 0$ we have $x_{n+m} = \widetilde{x}_m$. Clearly the forbidden patterns of $X_{\texttt{time}}$ can be recognized by a Turing machine with oracle $\WP(G)$.

Let $\widetilde{X}_{\texttt{time}} \subset \ag_X^{G \times \ZZ}$ be the periodic extension of ${X}_{\texttt{time}}$. That is, for all $\widetilde{t} \in \widetilde{X}_{\texttt{time}}$ and $g \in G$ we have $\widetilde{t}_{(g,k)} = \widetilde{t}_{(1_G,k)}$ and the configuration $x \in \ag_X^{\ZZ}$ defined by $x_k = \widetilde{t}_{(1_G,k)}$ belongs to $X_{\texttt{time}}$. 

Finally, we define $U \subset \widetilde{X}_{\texttt{time}} \times \{0,1,2\}^{G \times \ZZ}$ by a set of forbidden patterns. In order to describe this set, we denote by $\pi_1$ and $\pi_2$ the projections to the first and second coordinate respectively.

\begin{itemize}
	\item Let $(k_n)_{n \geq 1}$ be the sequence of positions in $\widetilde{x}$ such that $\widetilde{x}_{k_n} = \oplus$. Recall that $\FF_n$ denotes the set of forbidden patterns defining $Y_n$. We forbid all patterns $p$ with support $F \ni (1_G,0)$ such that $\pi_1(p)_{(1_G,0)} = \star$ and for which there is $n \in \NN$ such that the restriction of $\pi_2(p)$ to $F_n = \{ (g,k_n) \mid (g,k_n) \in F \}$ contains a pattern in $\FF_n$.
	\item We forbid all patterns $p$ with support $F = \{(1_G,0), (1_G,1) \}$ such that $\pi_1(p)_{(1_G,1)} \in \{ \vartriangleright, \bullet \}$ and $\pi_2(p)_{(1_G,1)} \neq \pi_2(p)_{(1_G,0)}$.
	\item For $s \in S$, we forbid all patterns with support  $F_s = \{(1_G,0), (s,1) \}$ such that $\pi_1(p)_{(s,1)} = s$ and $\pi_2(p)_{(s,1)} \neq \pi_2(p)_{(1_G,0)}$.
\end{itemize}

In other words, these patterns use the information on the first coordinate to force a structure on the second one as follows: The $n$-th coordinate marked with $\oplus$ after a $\star$ must carry a configuration $y \in Y_n$ in the second coordinate. The symbols $ \vartriangleright$ and $\bullet$ in the layer $(G,m)$ just copy the configuration in the layer $(G,m-1)$. The symbols from $S$ shift the whole configuration by $s \in S$. See Figure~\ref{figure.patate}.

\begin{figure}[h!]
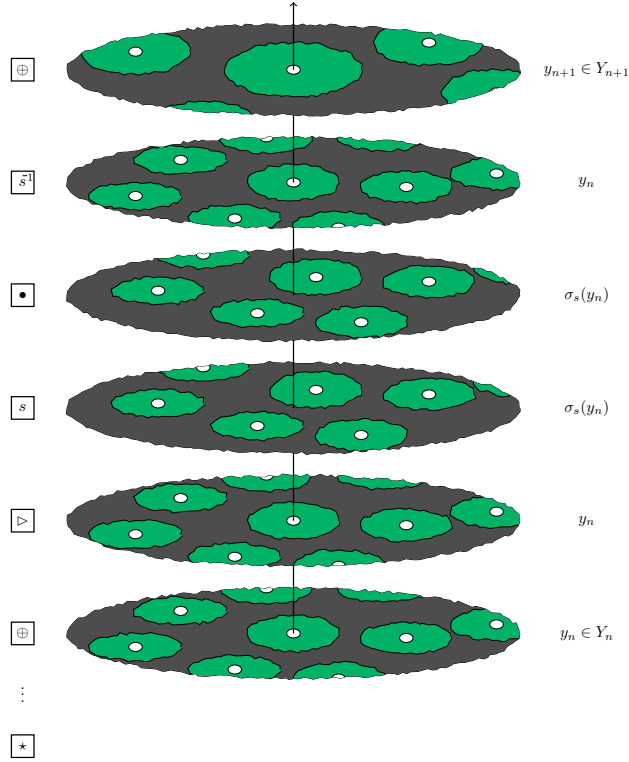

	\centering
	\include{patate}
	\caption{A typical configuration in $U\subseteq \left(\{ \bullet , \star, \oplus, \vartriangleright \} \cup S\right)\times \{ 0,1,2\}^{G\times\ZZ}$. Symbols on the left side of the picture correspond to the first coordinate of the configuration, and the part in $\{ 0,1,2\}^{G\times\ZZ}$ is on the right. On the example, the bottom $\oplus$ is the $n$-th appearence after $\star$.}
	\label{figure.patate}
\end{figure}

\begin{claim*}
	$U$ is a non-empty, $G \times \ZZ$-effectively closed subshift.
\end{claim*}
\begin{proof}
The first set of forbidden patterns is recursively enumerable with oracle $\WP(G)$ as $(k_n)$ is computable and $Y_n$ is $G$-effectively closed (the Turing machine accepting patterns of $Y_n$ can be constructed universally for all $(Y_n)_{n \in \NN}$ such that it receives $n \in \NN$, $p \in \{0,1,2\}^G$ as an input and accepts if $[p] \cap Y_n = \emptyset$). The rest of the forbidden patterns is a finite set, therefore $U$ is a $G \times \ZZ$-effectively closed subshift. It is non-empty as each $Y_n$ is non-empty.
\end{proof}

Now that the description of $U$ is done, we are ready to show Theorem~\ref{Teorema_simulacion}.

\begin{proof}
Let $\ag$ be the alphabet of $X$ and $T$ be the $G$-machine which on entry $p \in \ag^*_G$ accepts if and only if $[p]\cap X = \emptyset$. Using $\mathcal{M}_{\text{visit}}$ from Theorem~\ref{theorem_oracle_is_G_effective} we can construct from $T$ a machine $\widetilde{T}$ working on an infinite configuration whose description is as follows.

The machine $\widetilde{T}$ contains two tapes: a reading tape which is never modified and initially filled with symbols from $\ag$, and a working tape. The machine $\widetilde{T}$ iterates infinitely for $n = 1,2,\dots $ as follows: for $n \in \NN$, the machine iterates in order $k = 1,2,\dots,n$ the following procedure:

\begin{itemize}
	\item Copy the pattern appearing in the reading tape in the support $B(1_G,k)$ around the head to the working tape.
	\item Run $T$ over this pattern $n$ steps. If $T$ accepts at some point, then $\widetilde{T}$ accepts.
	\item Erase everything in the working tape and go back to the starting position.
\end{itemize}

Let $\Sigma \ni \sqcup$ be the alphabet of the working tape of $\widetilde{T}$ and let its set of states be $Q = \{1,\dots,k\}$, where $1$ is the initial state and $k$ the only accepting state. We proceed similarly to Theorem~\ref{theorem_group_undecidable_DP} by modeling this machine as a subshift on $G \times \ZZ$. We define the alphabet $\bg = \ag \times \Sigma \times \{0,\dots,k\}$. Here $\ag$ is the alphabet of $X$, $\Sigma$ is the alphabet of the working tape and $\{0,\dots,k\}$ codes the state of the head of a $G$-machine, $0$ coding the absence of a head. In order to describe the finite set of forbidden patterns we introduce some notation. Recall that $U$ is defined over the alphabet $\{\bullet,\star,\oplus,\vartriangleright \} \times \{0,1,2\}$. Therefore the set of forbidden patterns $\FF$ is defined over the alphabet  $\ag_{\text{Final}}$ where:

$$\ag_{\text{Final}} = \{\bullet,\star,\oplus,\vartriangleright \} \times \{0,1,2\} \times \ag \times \Sigma \times \{0,\dots,k\}.$$

We denote the projection to each of these five coordinates by $\pi_1,\dots,\pi_5$ respectively. The forbidden patterns in $\FF$ belong to four categories: \emph{configuration patterns}, \emph{starting patterns}, \emph{ending patterns} and \emph{transitions patterns}.

The \emph{configuration patterns} force that every $\ZZ$-coset sees the same symbol in the third coordinate. Said otherwise, the third coordinate is invariant under the action of $\ZZ$. To obtain this we forbid all $p$ with support $\{(1_G,0), (1_G,1)\}$ such that $\pi_3(p_{(1_G,0)})\neq \pi_3(p_{(1_G,1)})$.

The \emph{starting patterns} are defined by forbidding symbols in $\ag_{\text{Final}}$ in a way such that every time the symbol $\vartriangleright$ appears in a $G$-coset, then the working tape symbols are empty (that is, marked by $\sqcup$) and all positions marked by $1$ carry a head with the initial state. Formally, we force that all $a \in \ag_{\text{Final}}$ such that $\pi_1(a) = \ \vartriangleright$ must also satisfy $\pi_4(a) = \sqcup$. Furthermore, if $\pi_2(a) = 1$ then $\pi_5(a) = 1$ and if $\pi_2(a) \in \{0,2\}$ then $\pi_5(a)=0$.

The \emph{ending patterns} are described by forbidding the appearance of any symbol containing the accepting state $k$. Formally, every symbol $a \in \ag_{\text{Final}}$ $\pi_5(a) = k$ is forbidden.

The \emph{transition patterns} describe the evolution of $\widetilde{T}$ after a symbol $\vartriangleright$. Each time the symbol $\bullet$ appears it marks that the $G$-machines must execute one step with respect to the previous $G$-coset. The description of these patterns is the same as the one done in Theorem~\ref{theorem_group_undecidable_DP} with one difference. We update the tape according to the transition function of $\widetilde{T}$ only if a head is lying in a position not marked by a $0$ in the second coordinate. If this happens, then the tape does not evolve. 

Finally, we describe the $1$-block code $\phi$. Let $\bar{a} \in \ag$ be a symbol such that $\bar{a}^G \in X$. We define a local function $\Phi : \ag_{\text{Final}} \to \ag$ by:

$$ \Phi(a) = \begin{cases}
\pi_3(a) & \text{ if } \pi_1(a) = \star\\
\bar{a} & \text{ otherwise}
\end{cases}$$
and we set $\phi(x)_{(g,k)} = \Phi(x_{(g,k)})$. 

Let $x \in \ag^G$ be the $G$-projective subdynamics of $\phi(z)$, where $z \in U \times \bg^{G \times \ZZ}$ and avoids all forbidden patterns in $\FF$. By definition of $U$, as $\widetilde{X}_{\texttt{time}}$ is a periodic extension, each $G$-coset of $z$ is either completely marked by $\star$ or does not contain a $\star$ at all. If this last case happens, then $x = \bar{a}^G \in X$. Otherwise $\pi_1(z)_{(g,0)} = \star$ and thus by definition of $U$ we have $\pi_1(z)_{(g,k)} = \widetilde{x}_k$. Suppose $x \notin X$, then there exists a ball $B_n$ and $p \in \ag^{B_n}$ such that $[p] \cap X = \emptyset$. This implies that $T$ accepts the entry $p$ in a finite number of steps $n_{T}$. By definition, $\widetilde{T}$ also accepts all configurations in $[p]$ in a number of steps bounded by a function of $n_T$. Let $B_m$ be a ball such that $\widetilde{T}$ never leaves $B_m$ when working on $[p]$ (one could take for instance $m$ as the bound on the number of steps). Let $N \geq \max\{ |Q|,|S|, |\Sigma|, m  \}$. 
Then we know that $\widetilde{T}$ starting on position $1_G$ would accept an entry in $[p]$ in less than $\texttt{time}(N)$ steps. Consider $k_N$ the position of the $N$-th appearance of~$\otimes$ in $\widetilde{x}$. By definition we know that in the $G$-coset in $k_N$, the second coordinate contains a configuration $y \in \{0,1,2\}^G$ such that $y \in Y_N$. Therefore, there exists $g \in B(1_G,4N)$ such that $y_g =1$. As each word of length smaller or equal to $4N$ appears, then a codification of $g^{-1}$ eventually does. Using the rules of $U$, this means that after this word the next coset is marked by $\vartriangleright$, and the configuration in the second coordinate is $y' =\sigma_{g^{-1}(y)}$ thus $y'_{1_G} = 1$. By definition of $\widetilde{x}$, the next $\texttt{time}(N)$ cosets are marked by $\bullet$ thus simulating $\widetilde{T}$ for that number of steps as long as the head does not see a $0$ in the second coordinate. As there is a ball of size at least $N$ around the identity marked by a symbol $2$, then $\widetilde{T}$ is run for $\texttt{time(N)}$ steps, thus reaching the accepting state $k$ which is forbidden. This contradicts that $x \notin X$.

Conversely, each $x \in X$ can be obtained by constructing a configuration $z$ such that $\pi_3(z)_{(g,k)} = x_g$ and $\pi_1(z)_{(g,0)} = \star$. By definition of $\widetilde{T}$ and similar arguments as above, this configuration can be completed for all $g \in G$ and $k \geq 0$ without producing forbidden patterns. For $k \leq 0$ we can just fill the coordinate $(g,k)$ with the symbol $(\bullet, 0,x_g, \sqcup,0)$ without creating forbidden patterns.\end{proof}

We remark that the condition that $X$ must contain a uniform configuration can easily be replaced by weaker statements. For example, it suffices to contain a periodic configuration or more generally, a $G$-SFT $Y$ such that $Y \subset X$. In the proof above it would suffice to add a $\ZZ$-periodic extension of $Y$ as an extra coordinate and change the definition of the 1-block code $\phi$ such that it projects to this coordinate instead of $\bar{a}$.

Another interesting aspect of this construction is that even if the subshift $U$ is $G \times \ZZ$-effectively closed in general, it can sometimes be forced to be a sofic subshift. For example, if $G = \ZZ^d$ then $X_{\texttt{time}}$ is an effectively closed $\ZZ$-subshift and thus its periodic extension is a sofic $\ZZ^{d+1}$-subshift by~\cite{AubrunSablik2010,DBLP:conf/birthday/DurandRS10}. Also, we remark that in the second coordinate of $U$, it suffices to contain a non-empty subsystem of $Y_n$ in each $G$-coset. For $\ZZ^d$ it is not hard to produce sofic subshifts with those properties. For example, the subshift shown in Figure~\ref{figure.example_rb} in which each horizontal strip contains a periodic configuration which doubles its period when advancing vertically can be easily shown to be sofic and adapted by adding extra symbols to produce a suitable subsystem of the second layer of $U$.

\begin{figure}[h!]
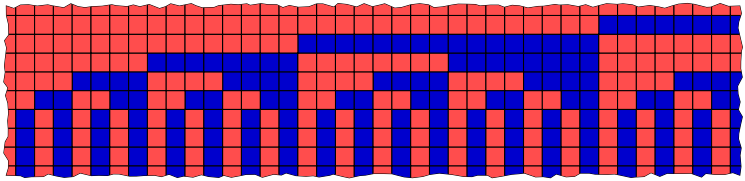

	\centering
	\include{example_rb}
	\caption{A sofic subshift which doubles its period.}
	\label{figure.example_rb}
\end{figure}


\bibliographystyle{plain}

\end{document}

%% file: example_config_mirror.tex
\begin{tikzpicture}[scale=0.15]

\clip[draw,decorate,decoration={random steps, segment length=3pt, amplitude=1pt}] (-9.5,0.5) rectangle (9.5,19.5); 

\foreach \y in {0,...,20} {
\lettre{0}{\y}{rouge}
}
\lettre{-10}{0}{black}
\lettre{-9}{0}{black}
\lettre{-8}{0}{white}
\lettre{-7}{0}{black}
\lettre{-6}{0}{black}
\lettre{-5}{0}{white}
\lettre{-4}{0}{black}
\lettre{-3}{0}{black}
\lettre{-2}{0}{white}
\lettre{-1}{0}{white}
\lettre{10}{0}{black}
\lettre{9}{0}{black}
\lettre{8}{0}{white}
\lettre{7}{0}{black}
\lettre{6}{0}{black}
\lettre{5}{0}{white}
\lettre{4}{0}{black}
\lettre{3}{0}{black}
\lettre{2}{0}{white}
\lettre{1}{0}{white}
\lettre{-10}{1}{black}
\lettre{-9}{1}{white}
\lettre{-8}{1}{white}
\lettre{-7}{1}{white}
\lettre{-6}{1}{black}
\lettre{-5}{1}{white}
\lettre{-4}{1}{white}
\lettre{-3}{1}{black}
\lettre{-2}{1}{white}
\lettre{-1}{1}{white}
\lettre{10}{1}{black}
\lettre{9}{1}{white}
\lettre{8}{1}{white}
\lettre{7}{1}{white}
\lettre{6}{1}{black}
\lettre{5}{1}{white}
\lettre{4}{1}{white}
\lettre{3}{1}{black}
\lettre{2}{1}{white}
\lettre{1}{1}{white}

\lettre{-10}{2}{white}
\lettre{-9}{2}{white}
\lettre{-8}{2}{black}
\lettre{-7}{2}{white}
\lettre{-6}{2}{white}
\lettre{-5}{2}{white}
\lettre{-4}{2}{black}
\lettre{-3}{2}{white}
\lettre{-2}{2}{white}
\lettre{-1}{2}{black}
\lettre{10}{2}{white}
\lettre{9}{2}{white}
\lettre{8}{2}{black}
\lettre{7}{2}{white}
\lettre{6}{2}{white}
\lettre{5}{2}{white}
\lettre{4}{2}{black}
\lettre{3}{2}{white}
\lettre{2}{2}{white}
\lettre{1}{2}{black}

\lettre{-10}{3}{white}
\lettre{-9}{3}{white}
\lettre{-8}{3}{white}
\lettre{-7}{3}{white}
\lettre{-6}{3}{black}
\lettre{-5}{3}{white}
\lettre{-4}{3}{black}
\lettre{-3}{3}{black}
\lettre{-2}{3}{white}
\lettre{-1}{3}{black}
\lettre{10}{3}{white}
\lettre{9}{3}{white}
\lettre{8}{3}{white}
\lettre{7}{3}{white}
\lettre{6}{3}{black}
\lettre{5}{3}{white}
\lettre{4}{3}{black}
\lettre{3}{3}{black}
\lettre{2}{3}{white}
\lettre{1}{3}{black}

\lettre{-10}{4}{black}
\lettre{-9}{4}{white}
\lettre{-8}{4}{black}
\lettre{-7}{4}{white}
\lettre{-6}{4}{white}
\lettre{-5}{4}{black}
\lettre{-4}{4}{black}
\lettre{-3}{4}{white}
\lettre{-2}{4}{black}
\lettre{-1}{4}{white}
\lettre{10}{4}{black}
\lettre{9}{4}{white}
\lettre{8}{4}{black}
\lettre{7}{4}{white}
\lettre{6}{4}{white}
\lettre{5}{4}{black}
\lettre{4}{4}{black}
\lettre{3}{4}{white}
\lettre{2}{4}{black}
\lettre{1}{4}{white}

\lettre{-10}{5}{white}
\lettre{-9}{5}{black}
\lettre{-8}{5}{black}
\lettre{-7}{5}{white}
\lettre{-6}{5}{black}
\lettre{-5}{5}{white}
\lettre{-4}{5}{black}
\lettre{-3}{5}{black}
\lettre{-2}{5}{black}
\lettre{-1}{5}{white}
\lettre{10}{5}{white}
\lettre{9}{5}{black}
\lettre{8}{5}{black}
\lettre{7}{5}{white}
\lettre{6}{5}{black}
\lettre{5}{5}{white}
\lettre{4}{5}{black}
\lettre{3}{5}{black}
\lettre{2}{5}{black}
\lettre{1}{5}{white}

\lettre{-10}{6}{black}
\lettre{-9}{6}{white}
\lettre{-8}{6}{white}
\lettre{-7}{6}{white}
\lettre{-6}{6}{white}
\lettre{-5}{6}{black}
\lettre{-4}{6}{white}
\lettre{-3}{6}{white}
\lettre{-2}{6}{white}
\lettre{-1}{6}{black}
\lettre{10}{6}{black}
\lettre{9}{6}{white}
\lettre{8}{6}{white}
\lettre{7}{6}{white}
\lettre{6}{6}{white}
\lettre{5}{6}{black}
\lettre{4}{6}{white}
\lettre{3}{6}{white}
\lettre{2}{6}{white}
\lettre{1}{6}{black}

\lettre{-10}{7}{white}
\lettre{-9}{7}{black}
\lettre{-8}{7}{black}
\lettre{-7}{7}{black}
\lettre{-6}{7}{white}
\lettre{-5}{7}{white}
\lettre{-4}{7}{black}
\lettre{-3}{7}{black}
\lettre{-2}{7}{white}
\lettre{-1}{7}{white}
\lettre{10}{7}{white}
\lettre{9}{7}{black}
\lettre{8}{7}{black}
\lettre{7}{7}{black}
\lettre{6}{7}{white}
\lettre{5}{7}{white}
\lettre{4}{7}{black}
\lettre{3}{7}{black}
\lettre{2}{7}{white}
\lettre{1}{7}{white}

\lettre{-10}{8}{white}
\lettre{-9}{8}{white}
\lettre{-8}{8}{black}
\lettre{-7}{8}{black}
\lettre{-6}{8}{black}
\lettre{-5}{8}{black}
\lettre{-4}{8}{black}
\lettre{-3}{8}{white}
\lettre{-2}{8}{black}
\lettre{-1}{8}{white}
\lettre{10}{8}{white}
\lettre{9}{8}{white}
\lettre{8}{8}{black}
\lettre{7}{8}{black}
\lettre{6}{8}{black}
\lettre{5}{8}{black}
\lettre{4}{8}{black}
\lettre{3}{8}{white}
\lettre{2}{8}{black}
\lettre{1}{8}{white}

\lettre{-10}{9}{white}
\lettre{-9}{9}{white}
\lettre{-8}{9}{black}
\lettre{-7}{9}{white}
\lettre{-6}{9}{white}
\lettre{-5}{9}{white}
\lettre{-4}{9}{white}
\lettre{-3}{9}{white}
\lettre{-2}{9}{black}
\lettre{-1}{9}{black}
\lettre{10}{9}{white}
\lettre{9}{9}{white}
\lettre{8}{9}{black}
\lettre{7}{9}{white}
\lettre{6}{9}{white}
\lettre{5}{9}{white}
\lettre{4}{9}{white}
\lettre{3}{9}{white}
\lettre{2}{9}{black}
\lettre{1}{9}{black}

\lettre{-10}{10}{white}
\lettre{-9}{10}{white}
\lettre{-8}{10}{white}
\lettre{-7}{10}{black}
\lettre{-6}{10}{black}
\lettre{-5}{10}{white}
\lettre{-4}{10}{white}
\lettre{-3}{10}{black}
\lettre{-2}{10}{black}
\lettre{-1}{10}{black}
\lettre{10}{10}{white}
\lettre{9}{10}{white}
\lettre{8}{10}{white}
\lettre{7}{10}{black}
\lettre{6}{10}{black}
\lettre{5}{10}{white}
\lettre{4}{10}{white}
\lettre{3}{10}{black}
\lettre{2}{10}{black}
\lettre{1}{10}{black}

\lettre{-10}{11}{white}
\lettre{-9}{11}{black}
\lettre{-8}{11}{black}
\lettre{-7}{11}{black}
\lettre{-6}{11}{white}
\lettre{-5}{11}{white}
\lettre{-4}{11}{black}
\lettre{-3}{11}{black}
\lettre{-2}{11}{white}
\lettre{-1}{11}{white}
\lettre{10}{11}{white}
\lettre{9}{11}{black}
\lettre{8}{11}{black}
\lettre{7}{11}{black}
\lettre{6}{11}{white}
\lettre{5}{11}{white}
\lettre{4}{11}{black}
\lettre{3}{11}{black}
\lettre{2}{11}{white}
\lettre{1}{11}{white}

\lettre{-10}{12}{white}
\lettre{-9}{12}{white}
\lettre{-8}{12}{white}
\lettre{-7}{12}{white}
\lettre{-6}{12}{black}
\lettre{-5}{12}{white}
\lettre{-4}{12}{white}
\lettre{-3}{12}{black}
\lettre{-2}{12}{white}
\lettre{-1}{12}{white}
\lettre{10}{12}{white}
\lettre{9}{12}{white}
\lettre{8}{12}{white}
\lettre{7}{12}{white}
\lettre{6}{12}{black}
\lettre{5}{12}{white}
\lettre{4}{12}{white}
\lettre{3}{12}{black}
\lettre{2}{12}{white}
\lettre{1}{12}{white}

\lettre{-10}{13}{black}
\lettre{-9}{13}{black}
\lettre{-8}{13}{black}
\lettre{-7}{13}{black}
\lettre{-6}{13}{white}
\lettre{-5}{13}{white}
\lettre{-4}{13}{black}
\lettre{-3}{13}{black}
\lettre{-2}{13}{white}
\lettre{-1}{13}{black}
\lettre{10}{13}{black}
\lettre{9}{13}{black}
\lettre{8}{13}{black}
\lettre{7}{13}{black}
\lettre{6}{13}{white}
\lettre{5}{13}{white}
\lettre{4}{13}{black}
\lettre{3}{13}{black}
\lettre{2}{13}{white}
\lettre{1}{13}{black}

\lettre{-10}{14}{white}
\lettre{-9}{14}{white}
\lettre{-8}{14}{black}
\lettre{-7}{14}{black}
\lettre{-6}{14}{black}
\lettre{-5}{14}{white}
\lettre{-4}{14}{white}
\lettre{-3}{14}{black}
\lettre{-2}{14}{black}
\lettre{-1}{14}{black}
\lettre{10}{14}{white}
\lettre{9}{14}{white}
\lettre{8}{14}{black}
\lettre{7}{14}{black}
\lettre{6}{14}{black}
\lettre{5}{14}{white}
\lettre{4}{14}{white}
\lettre{3}{14}{black}
\lettre{2}{14}{black}
\lettre{1}{14}{black}

\lettre{-10}{15}{white}
\lettre{-9}{15}{white}
\lettre{-8}{15}{black}
\lettre{-7}{15}{white}
\lettre{-6}{15}{white}
\lettre{-5}{15}{white}
\lettre{-4}{15}{black}
\lettre{-3}{15}{white}
\lettre{-2}{15}{black}
\lettre{-1}{15}{white}
\lettre{10}{15}{white}
\lettre{9}{15}{white}
\lettre{8}{15}{black}
\lettre{7}{15}{white}
\lettre{6}{15}{white}
\lettre{5}{15}{white}
\lettre{4}{15}{black}
\lettre{3}{15}{white}
\lettre{2}{15}{black}
\lettre{1}{15}{white}

\lettre{-10}{16}{black}
\lettre{-9}{16}{black}
\lettre{-8}{16}{white}
\lettre{-7}{16}{black}
\lettre{-6}{16}{black}
\lettre{-5}{16}{white}
\lettre{-4}{16}{white}
\lettre{-3}{16}{white}
\lettre{-2}{16}{white}
\lettre{-1}{16}{black}
\lettre{10}{16}{black}
\lettre{9}{16}{black}
\lettre{8}{16}{white}
\lettre{7}{16}{black}
\lettre{6}{16}{black}
\lettre{5}{16}{white}
\lettre{4}{16}{white}
\lettre{3}{16}{white}
\lettre{2}{16}{white}
\lettre{1}{16}{black}

\lettre{-10}{17}{white}
\lettre{-9}{17}{black}
\lettre{-8}{17}{white}
\lettre{-7}{17}{black}
\lettre{-6}{17}{white}
\lettre{-5}{17}{white}
\lettre{-4}{17}{black}
\lettre{-3}{17}{black}
\lettre{-2}{17}{white}
\lettre{-1}{17}{white}
\lettre{10}{17}{white}
\lettre{9}{17}{black}
\lettre{8}{17}{white}
\lettre{7}{17}{black}
\lettre{6}{17}{white}
\lettre{5}{17}{white}
\lettre{4}{17}{black}
\lettre{3}{17}{black}
\lettre{2}{17}{white}
\lettre{1}{17}{white}

\lettre{-10}{18}{black}
\lettre{-9}{18}{black}
\lettre{-8}{18}{white}
\lettre{-7}{18}{white}
\lettre{-6}{18}{black}
\lettre{-5}{18}{black}
\lettre{-4}{18}{white}
\lettre{-3}{18}{white}
\lettre{-2}{18}{white}
\lettre{-1}{18}{black}
\lettre{10}{18}{black}
\lettre{9}{18}{black}
\lettre{8}{18}{white}
\lettre{7}{18}{white}
\lettre{6}{18}{black}
\lettre{5}{18}{black}
\lettre{4}{18}{white}
\lettre{3}{18}{white}
\lettre{2}{18}{white}
\lettre{1}{18}{black}

\lettre{-10}{19}{white}
\lettre{-9}{19}{white}
\lettre{-8}{19}{black}
\lettre{-7}{19}{white}
\lettre{-6}{19}{black}
\lettre{-5}{19}{white}
\lettre{-4}{19}{white}
\lettre{-3}{19}{black}
\lettre{-2}{19}{black}
\lettre{-1}{19}{black}
\lettre{10}{19}{white}
\lettre{9}{19}{white}
\lettre{8}{19}{black}
\lettre{7}{19}{white}
\lettre{6}{19}{black}
\lettre{5}{19}{white}
\lettre{4}{19}{white}
\lettre{3}{19}{black}
\lettre{2}{19}{black}
\lettre{1}{19}{black}

\lettre{-10}{20}{black}
\lettre{-9}{20}{white}
\lettre{-8}{20}{white}
\lettre{-7}{20}{black}
\lettre{-6}{20}{black}
\lettre{-5}{20}{black}
\lettre{-4}{20}{white}
\lettre{-3}{20}{black}
\lettre{-2}{20}{white}
\lettre{-1}{20}{white}
\lettre{10}{20}{black}
\lettre{9}{20}{white}
\lettre{8}{20}{white}
\lettre{7}{20}{black}
\lettre{6}{20}{black}
\lettre{5}{20}{black}
\lettre{4}{20}{white}
\lettre{3}{20}{black}
\lettre{2}{20}{white}
\lettre{1}{20}{white}

\end{tikzpicture} \ \ \ \ \ \begin{tikzpicture}[scale=0.1]

\draw (0,-2) node{$y_1\in X_\text{mirror}$};
\draw[color=vert!20,fill=vert!20,decorate,decoration={random steps,segment length=5pt,amplitude=2pt}] (-10,0) rectangle (10,10);
\draw[color=rouge,fill=rouge] (-0.25,0) rectangle (0.25,10);
\draw (-0.25,0) rectangle (-0.25,10);
\draw (0.25,0) rectangle (0.25,10); 

\draw[fill=vert!50] (4,3) rectangle (8,7);
\draw (6,5) node{$P_1$};
\draw[fill=vert!50] (-4,3) rectangle (-8,7);
\draw (-6,5) node{$\widetilde{P}_1$}; 

\begin{scope}[shift={(25,0)}]
\draw (0,-2) node{$y_2\in X_\text{mirror}$};
\draw[color=bleu!20,fill=bleu!20,decorate,decoration={random steps,segment length=5pt,amplitude=2pt}] (-10,0) rectangle (10,10); 
\draw[color=rouge,fill=rouge] (-0.25,0) rectangle (0.25,10);
\draw (-0.25,0) rectangle (-0.25,10);
\draw (0.25,0) rectangle (0.25,10);  

\draw[fill=bleu!50] (4,3) rectangle (8,7);
\draw (6,5) node{$P_2$}; 
\draw[fill=bleu!50] (-4,3) rectangle (-8,7);
\draw (-6,5) node{$\widetilde{P}_2$};

\end{scope} 

\begin{scope}[shift={(50,0)}]
\draw (0,-2) node{$\widetilde{y}\notin X_\text{mirror}$};
\draw[color=vert!20,fill=vert!20,decorate,decoration={random steps,segment length=5pt,amplitude=2pt}] (-10,0) rectangle (10,10); 
\draw[color=rouge,fill=rouge] (-0.25,0) rectangle (0.25,10);
\draw (-0.25,0) rectangle (-0.25,10);
\draw (0.25,0) rectangle (0.25,10);  

\draw[fill=bleu!50] (4,3) rectangle (8,7);
\draw (6,5) node{$P_2$}; 
\draw[fill=vert!50] (-4,3) rectangle (-8,7);
\draw (-6,5) node{$\widetilde{P}_1$}; 

\end{scope}

\begin{scope}[shift={(0,15)}]
\draw (0,12) node{$x_1\in X$};
\draw[color=vert!50,fill=vert!50,decorate,decoration={random steps,segment length=5pt,amplitude=2pt}] (-10,0) rectangle (10,10); 

\draw[fill=orange!50] (3,2) rectangle (9,8); 
\draw[fill=vert!50] (4,3) rectangle (8,7);
\draw (6,5) node{$Q_1$}; 

\end{scope}

\begin{scope}[shift={(25,15)}]
\draw (0,12) node{$x_2\in X$};
\draw[color=bleu!50,fill=bleu!50,decorate,decoration={random steps,segment length=5pt,amplitude=2pt}] (-10,0) rectangle (10,10); 

\draw[fill=orange!50] (3,2) rectangle (9,8); 
\draw[fill=bleu!50] (4,3) rectangle (8,7);
\draw (6,5) node{$Q_2$};  

\end{scope}

\begin{scope}[shift={(50,15)}]
\draw (0,12) node{$\widetilde{x}\in X$};
\draw[color=vert!50,fill=vert!50,decorate,decoration={random steps,segment length=5pt,amplitude=2pt}] (-10,0) rectangle (10,10); 

\draw[fill=orange!50] (3,2) rectangle (9,8); 
\draw[fill=bleu!50] (4,3) rectangle (8,7);
\draw (6,5) node{$Q_2$};  

\end{scope}

\draw (0,12) node{$\downarrow \phi$};
\draw (25,12) node{$\downarrow \phi$};
\draw (50,12) node{$\downarrow \phi$};

\end{tikzpicture}

%% file: simulation_g_machine.tex
\begin{tikzpicture}[scale=0.65]

\begin{scope}[scale = 0.75, shift={(-5.5,0)},rotate=0]
\def \c{0.5}
\def \b{0.3}
\def \a{0.2}

\filldraw[dashed,fill=blue,fill opacity=0.3] (-1.4,-2/3) -- (0.4,-2/3)-- (1.4,2/3)  -- (-0.4,2/3)  -- cycle;

\draw [->] (0,0) to (-5,0);
\draw [->] (0,0) to (5,0);
\draw [->] (0,0) to (-3,-4);
\draw [->] (0,0) to (3,4);

\node at (1.5,0.3) {$a$};
\node at (0.7,1.6) {$b$};

\draw [->] (-3,0) to (-4,-4/3);
\draw [->] (-3,0) to (-2, 4/3);

\draw [->] (3,0) to (2,-4/3);
\draw [->] (3,0) to (4, 4/3);

\draw [->] (-2,-8/3) to (-11/3,-8/3);
\draw [->] (-2,-8/3) to (-1/3,-8/3);

\draw [->] (2,8/3) to (11/3,8/3);
\draw [->] (2,8/3) to (1/3,8/3);

\draw[fill = black] (0,0) circle (\c);
\draw[fill = black] (-3,0) circle (\b);
\draw[fill = white] (3,0) circle (\b);
\draw[fill = black] (2,8/3) circle (\b);
\draw[fill = white] (-2,-8/3) circle (\b);

\draw [->] (-4,0) to (-3.5,4/6);
\draw [->] (-4,0) to (-4.5,-4/6);
\draw[fill = white] (-4,0) circle (\a);
\node at (-4,0) {\fontsize{0.4}{0.1}\textbf{$\sqcup$}};
\draw [->] (-3.5,-2/3) to (-4.3,-2/3);
\draw [->] (-3.5,-2/3) to (-2.7,-2/3);
\draw[fill = white] (-3.5,-2/3) circle (\a);
\node at (-3.5,-2/3) {\fontsize{0.4}{0.1}\textbf{$\sqcup$}};
\draw [->] (-2.5,2/3) to (-3.3,2/3);
\draw [->] (-2.5,2/3) to (-1.7,2/3);
\draw[fill = white] (-2.5,2/3) circle (\a);
\node at (-2.5,2/3) {\fontsize{0.4}{0.1}\textbf{$\sqcup$}};

\draw [->] (4,0) to (3.5,-4/6);
\draw [->] (4,0) to (4.5,4/6);
\draw[fill = white] (4,0) circle (\a);
\node at (4,0) {\fontsize{0.4}{0.1}\textbf{$\sqcup$}};
\draw [->] (3.5,2/3) to (4.3,2/3);
\draw [->] (3.5,2/3) to (2.7,2/3);
\draw[fill = white] (3.5,2/3) circle (\a);
\node at (3.5,2/3)  {\fontsize{0.4}{0.1}\textbf{$\sqcup$}};
\draw [->] (2.5,-2/3) to (3.3,-2/3);
\draw [->] (2.5,-2/3) to (1.7,-2/3);
\draw[fill = white] (2.5,-2/3) circle (\a);
\node at (2.5,-2/3){\fontsize{0.4}{0.1}\textbf{$\sqcup$}};

\draw [->] (-2.5,-10/3) to (-3.3,-10/3);
\draw [->] (-2.5,-10/3) to (-1.7,-10/3);
\draw[fill = white] (-2.5,-10/3) circle (\a);
\node at (-2.5,-10/3) {\fontsize{0.4}{0.1}\textbf{$\sqcup$}};
\draw [->] (-3,-8/3) to (-3.5,-10/3);
\draw [->] (-3,-8/3) to (-2.5,-6/3);
\draw[fill = white] (-3,-8/3) circle (\a);
\node at  (-3,-8/3){\fontsize{0.4}{0.1}\textbf{$\sqcup$}};
\draw [->] (-1,-8/3) to (-1.5,-10/3);
\draw [->] (-1,-8/3) to (-0.5,-6/3);
\draw[fill = white] (-1,-8/3) circle (\a);
\node at (-1,-8/3){\fontsize{0.4}{0.1}\textbf{$\sqcup$}};

\draw [->] (2.5,10/3) to (3.3,10/3);
\draw [->] (2.5,10/3) to (1.7,10/3);
\draw[fill = white] (2.5,10/3) circle (\a);
\node at (2.5,10/3){\fontsize{0.4}{0.1}\textbf{$\sqcup$}};
\draw [->] (3,8/3) to (3.5,10/3);
\draw [->] (3,8/3) to (2.5,6/3);
\draw[fill = white] (3,8/3) circle (\a);
\node at (3,8/3) {\fontsize{0.4}{0.1}\textbf{$\sqcup$}};
\draw [->] (1,8/3) to (1.5,10/3);
\draw [->] (1,8/3) to (0.5,6/3);
\draw[fill = white] (1,8/3) circle (\a);
\node at (1,8/3) {\fontsize{0.4}{0.1}\textbf{$\sqcup$}};

\draw [->, ultra thick] (-2,3) to (-0.6, 0.6);
\draw[fill = white] (-2,3) circle (0.5);
\draw node (1) at (-2,3) {$q_1$};
\end{scope}

\begin{scope}[scale = 0.75, shift={(5.5,0)},rotate=0]
\def \c{0.5}
\def \b{0.3}
\def \a{0.2}
\begin{scope}[scale = 3/5, shift={(-5,0)},rotate=0]
\filldraw[dashed,fill=blue,fill opacity=0.3] (-1.4,-2/3) -- (0.4,-2/3)-- (1.4,2/3)  -- (-0.4,2/3)  -- cycle;
\end{scope}
\draw [->] (0,0) to (-5,0);
\draw [->] (0,0) to (5,0);
\draw [->] (0,0) to (-3,-4);
\draw [->] (0,0) to (3,4);

\node at (1.5,0.3) {$a$};
\node at (0.7,1.6) {$b$};

\draw [->] (-3,0) to (-4,-4/3);
\draw [->] (-3,0) to (-2, 4/3);

\draw [->] (3,0) to (2,-4/3);
\draw [->] (3,0) to (4, 4/3);

\draw [->] (-2,-8/3) to (-11/3,-8/3);
\draw [->] (-2,-8/3) to (-1/3,-8/3);

\draw [->] (2,8/3) to (11/3,8/3);
\draw [->] (2,8/3) to (1/3,8/3);

\draw[fill = white] (0,0) circle (\c);
\draw[fill = white] (-3,0) circle (\b);
\draw[fill = white] (3,0) circle (\b);
\node at (3,0) {{\footnotesize$\sqcup$}};
\draw[fill = white] (2,8/3) circle (\b);
\node at (2,8/3) {{\footnotesize$\sqcup$}};
\draw[fill = white] (-2,-8/3) circle (\b);
\node at (-2,-8/3) {{\footnotesize$\sqcup$}};

\draw [->] (-4,0) to (-3.5,4/6);
\draw [->] (-4,0) to (-4.5,-4/6);
\draw[fill = black] (-4,0) circle (\a);
\draw [->] (-3.5,-2/3) to (-4.3,-2/3);
\draw [->] (-3.5,-2/3) to (-2.7,-2/3);
\draw[fill = white] (-3.5,-2/3) circle (\a);
\draw [->] (-2.5,2/3) to (-3.3,2/3);
\draw [->] (-2.5,2/3) to (-1.7,2/3);
\draw[fill = black] (-2.5,2/3) circle (\a);

\draw [->] (4,0) to (3.5,-4/6);
\draw [->] (4,0) to (4.5,4/6);
\draw[fill = white] (4,0) circle (\a);
\node at (4,0) {\fontsize{0.4}{0.1}\textbf{$\sqcup$}};
\draw [->] (3.5,2/3) to (4.3,2/3);
\draw [->] (3.5,2/3) to (2.7,2/3);
\draw[fill = white] (3.5,2/3) circle (\a);
\node at (3.5,2/3)  {\fontsize{0.4}{0.1}\textbf{$\sqcup$}};
\draw [->] (2.5,-2/3) to (3.3,-2/3);
\draw [->] (2.5,-2/3) to (1.7,-2/3);
\draw[fill = white] (2.5,-2/3) circle (\a);
\node at (2.5,-2/3){\fontsize{0.4}{0.1}\textbf{$\sqcup$}};

\draw [->] (-2.5,-10/3) to (-3.3,-10/3);
\draw [->] (-2.5,-10/3) to (-1.7,-10/3);
\draw[fill = white] (-2.5,-10/3) circle (\a);
\node at (-2.5,-10/3) {\fontsize{0.4}{0.1}\textbf{$\sqcup$}};
\draw [->] (-3,-8/3) to (-3.5,-10/3);
\draw [->] (-3,-8/3) to (-2.5,-6/3);
\draw[fill = white] (-3,-8/3) circle (\a);
\node at  (-3,-8/3){\fontsize{0.4}{0.1}\textbf{$\sqcup$}};
\draw [->] (-1,-8/3) to (-1.5,-10/3);
\draw [->] (-1,-8/3) to (-0.5,-6/3);
\draw[fill = white] (-1,-8/3) circle (\a);
\node at (-1,-8/3){\fontsize{0.4}{0.1}\textbf{$\sqcup$}};

\draw [->] (2.5,10/3) to (3.3,10/3);
\draw [->] (2.5,10/3) to (1.7,10/3);
\draw[fill = white] (2.5,10/3) circle (\a);
\node at (2.5,10/3){\fontsize{0.4}{0.1}\textbf{$\sqcup$}};
\draw [->] (3,8/3) to (3.5,10/3);
\draw [->] (3,8/3) to (2.5,6/3);
\draw[fill = white] (3,8/3) circle (\a);
\node at (3,8/3) {\fontsize{0.4}{0.1}\textbf{$\sqcup$}};
\draw [->] (1,8/3) to (1.5,10/3);
\draw [->] (1,8/3) to (0.5,6/3);
\draw[fill = white] (1,8/3) circle (\a);
\node at (1,8/3) {\fontsize{0.4}{0.1}\textbf{$\sqcup$}};

\draw [->, ultra thick] (-2,3) to (-0.6, 0.6);
\draw[fill = white] (-2,3) circle (0.5);
\draw node (2) at (-2,3) {$q_2$};
\end{scope}
\path
(1) edge [thick, bend left=30,->] node[swap]  {} (2);

\draw node at (-1.5,4) {$\delta(q_1,\begin{tikzpicture}
\draw[fill = black] (0,0) circle (0.1);
\end{tikzpicture}) = (q_2,\begin{tikzpicture}
\draw[fill = white] (0,0) circle (0.1);
\end{tikzpicture},a)$};

\end{tikzpicture}

%% file: simulation_z2_machine.tex
\begin{tikzpicture}[scale=0.50]

\begin{scope}[scale = 0.75, shift={(-6,0)},rotate=0]
\def \c{0.3}

\draw [<->] (-5,0) to (5,0);
\draw [<->] (-6,-1) to (4,-1);
\draw [<->] (-7,-2) to (3,-2);
\draw [<->] (-4,1) to (6,1);
\draw [<->] (-3,2) to (7,2);

\draw [<->] (3,3) to (-3,-3);
\draw [<->] (3-5/3,3) to (-3-5/3,-3);
\draw [<->] (3-10/3,3) to (-3-10/3,-3);
\draw [<->] (3+5/3,3) to (-3+5/3,-3);
\draw [<->] (3+10/3,3) to (-3+10/3,-3);

\foreach \i in {-10/3,-5/3,0,5/3,10/3}{
	\foreach \j in {-2,-1,0,1,2}{
		\draw[fill = white] (\i+\j,\j) circle (\c);
	}
}


\filldraw[thick,dashed,fill=blue,fill opacity=0.3] (-5/6-0.5,-0.5) -- (-5/6+0.5,0.5) -- (5/6+0.5,0.5)  -- (5/6-0.5,-0.5)  -- cycle;

\draw[fill = white] (-10/3,0) circle (\c);
\draw[fill = black!75] (-5/3,0)circle (\c);
\draw[fill = black!75] (0,0)circle (\c);
\draw[fill = white] (5/3,0) circle (\c);
\draw[fill = white] (10/3,0) circle (\c);

\draw[fill = white] (-10/3+1,1) circle (\c);
\draw[fill = white] (-5/3+1,1)circle (\c);
\draw[fill = black!75] (1,1)circle (\c);
\draw[fill = black!75] (5/3+1,1) circle (\c);
\draw[fill = white] (10/3+1,1) circle (\c);

\draw[fill = white] (-10/3-1,-1) circle (\c);
\draw[fill = black!75] (-5/3-1,-1)circle (\c);
\draw[fill = black!75] (-1,-1)circle (\c);
\draw[fill = black!75] (5/3-1,-1) circle (\c);
\draw[fill = white] (10/3-1,-1) circle (\c);

\foreach \i in {-10/3,10/3}{
	\foreach \j in {-2,-1,0,1,2}{
		\node at (\i+\j,\j) {\fontsize{0.4}{0}\textbf{$\sqcup$}};
	}
}

\foreach \i in {-5/3,0,5/3}{
	\foreach \j in {-2,2}{
		\node at (\i+\j,\j) {\fontsize{0.4}{0}\textbf{$\sqcup$}};
	}
}

\draw [->, ultra thick] (-2,3) to (-0.3, 0.3);
\draw[fill = white] (-2,3) circle (0.5);
\draw node (1) at (-2,3) {$q_1$};

\end{scope}

\begin{scope}[scale = 0.75, shift={(6,0)},rotate=0]
\def \c{0.3}

\draw [<->] (-5,0) to (5,0);
\draw [<->] (-6,-1) to (4,-1);
\draw [<->] (-7,-2) to (3,-2);
\draw [<->] (-4,1) to (6,1);
\draw [<->] (-3,2) to (7,2);

\draw [<->] (3,3) to (-3,-3);
\draw [<->] (3-5/3,3) to (-3-5/3,-3);
\draw [<->] (3-10/3,3) to (-3-10/3,-3);
\draw [<->] (3+5/3,3) to (-3+5/3,-3);
\draw [<->] (3+10/3,3) to (-3+10/3,-3);

\filldraw[thick,dashed,fill=blue,fill opacity=0.3] (-5/6-0.5,-0.5) -- (-5/6+0.5,0.5) -- (5/6+0.5,0.5)  -- (5/6-0.5,-0.5)  -- cycle;

\foreach \i in {-10/3,-5/3,0,5/3,10/3}{
	\foreach \j in {-2,-1,0,1,2}{
		\draw[fill = white] (\i+\j,\j) circle (\c);
	}
}

\draw[fill = white] (-10/3,0) circle (\c);
\draw[fill = black!75] (-5/3,0)circle (\c);
\draw[fill = white] (5/3,0) circle (\c);
\draw[fill = white] (10/3,0) circle (\c);

\draw[fill = white] (-10/3+1,1) circle (\c);
\draw[fill = white] (-5/3+1,1)circle (\c);
\draw[fill = black!75] (1,1)circle (\c);
\draw[fill = black!75] (5/3+1,1) circle (\c);
\draw[fill = white] (10/3+1,1) circle (\c);

\draw[fill = white] (-10/3-1,-1) circle (\c);
\draw[fill = black!75] (-5/3-1,-1)circle (\c);
\draw[fill = black!75] (-1,-1)circle (\c);
\draw[fill = black!75] (5/3-1,-1) circle (\c);
\draw[fill = white] (10/3-1,-1) circle (\c);

\foreach \i in {-10/3,10/3}{
	\foreach \j in {-2,-1,0,1,2}{
		\node at (\i+\j,\j) {\fontsize{0.4}{0}\textbf{$\sqcup$}};
	}
}

\foreach \i in {-5/3,0,5/3}{
	\foreach \j in {-2,2}{
		\node at (\i+\j,\j) {\fontsize{0.4}{0}\textbf{$\sqcup$}};
	}
}

\draw [->, ultra thick] (-2+5/3,3) to (-0.3+5/3, 0.3);
\draw[fill = white] (-2+5/3,3) circle (0.5);
\draw node (2) at (-2+5/3,3) {$q_2$};
\end{scope}
\path
(1) edge [thick, bend left=20,->] node[swap]  {} (2);
\draw node at (-1.5,4) {$\delta(q_1,\begin{tikzpicture}
	\draw[fill = black] (0,0) circle (0.1);
	\end{tikzpicture}) = (q_2,\begin{tikzpicture}
	\draw[fill = white] (0,0) circle (0.1);
	\end{tikzpicture},(1,0))$};

\end{tikzpicture}

%% file: Z_machine_inside_G_bis.tex
\begin{tikzpicture}[scale=0.4]

\def \scale{0.6}

\begin{scope}[shift={(0,0)}, scale = \scale]
\def \c{0.5}
\def \b{0.3}
\def \a{0.2}

\begin{scope}[opacity=0.25]
\draw [->] (0,0) to (-5,0);
\draw [->] (0,0) to (5,0);
\draw [->] (0,0) to (-3,-4);
\draw [->] (0,0) to (3,4);


\draw [->] (-3,0) to (-4,-4/3);
\draw [->] (-3,0) to (-2, 4/3);

\draw [->] (3,0) to (2,-4/3);
\draw [->] (3,0) to (4, 4/3);

\draw [->] (-2,-8/3) to (-11/3,-8/3);
\draw [->] (-2,-8/3) to (-1/3,-8/3);

\draw [->] (2,8/3) to (11/3,8/3);
\draw [->] (2,8/3) to (1/3,8/3);


\draw [->] (-4,0) to (-3.5,4/6);
\draw [->] (-4,0) to (-4.5,-4/6);
\draw [->] (-3.5,-2/3) to (-4.3,-2/3);
\draw [->] (-3.5,-2/3) to (-2.7,-2/3);
\draw [->] (-2.5,2/3) to (-3.3,2/3);
\draw [->] (-2.5,2/3) to (-1.7,2/3);

\draw [->] (4,0) to (3.5,-4/6);
\draw [->] (4,0) to (4.5,4/6);
\draw [->] (3.5,2/3) to (4.3,2/3);
\draw [->] (3.5,2/3) to (2.7,2/3);
\draw [->] (2.5,-2/3) to (3.3,-2/3);
\draw [->] (2.5,-2/3) to (1.7,-2/3);

\draw [->] (-2.5,-10/3) to (-3.3,-10/3);
\draw [->] (-2.5,-10/3) to (-1.7,-10/3);
\draw [->] (-3,-8/3) to (-3.5,-10/3);
\draw [->] (-3,-8/3) to (-2.5,-6/3);
\draw [->] (-1,-8/3) to (-1.5,-10/3);
\draw [->] (-1,-8/3) to (-0.5,-6/3);

\draw [->] (2.5,10/3) to (3.3,10/3);
\draw [->] (2.5,10/3) to (1.7,10/3);
\draw [->] (3,8/3) to (3.5,10/3);
\draw [->] (3,8/3) to (2.5,6/3);
\draw [->] (1,8/3) to (1.5,10/3);
\draw [->] (1,8/3) to (0.5,6/3);

\end{scope}

\node at (0,0) {\large$\vartriangleright$};
\draw [-] (0.25,0) to (5,0);
\draw [-] (5,0) to (5.5, 4/6);
\draw [->] (5.5, 4/6) to (6.5, 4/6);

\path
(6, 5/6) edge [dashed, bend left=30,->] node[swap]  {} (8.5, 10/6);

\draw [-] (4,-4.5) to (-8,-4.5);
\draw [-] (4,-4.5) to (10,3.5);
\node at (-6,-3.5) {\scalebox{0.7}{Layer 6}};
\node at (2.5,-3.5) {\scalebox{0.7}{$\mathcal{M}_{SIM}$}};

\node at (9.5,1.5) {\rotatebox[]{330}{\scalebox{2}[3]{$\{$}}};

\begin{scope}[shift={(-1.5,0)}, scale = 1.2]

\draw [opacity = 0.25](11,2) rectangle (11.8,2.8);
\node at (11.4,2.4) {\scalebox{0.7}{$\vartriangleright$}};
\draw [opacity = 0.25](11.8,2) rectangle (12.6,2.8);
\node at (12.2,2.4) {\scalebox{0.7}{$($}};
\draw [opacity = 0.25](12.6,2) rectangle (13.4,2.8);
\node at (13,2.4) {\scalebox{0.7}{$\epsilon$}};
\draw [opacity = 0.25](13.4,2) rectangle (14.2,2.8);
\node at (13.8,2.2) {\scalebox{0.7}{$,$}};
\draw [opacity = 0.25](14.2,2) rectangle (15,2.8);
\node at (14.6,2.4) {\scalebox{0.7}{$\begin{tikzpicture}
\draw[preaction={fill, blue!25},pattern=north west lines] (0,0) circle (0.1);
\end{tikzpicture}$}};
\draw [opacity = 0.25](15,2) rectangle (15.8,2.8);
\node at (15.4,2.4) {\scalebox{0.7}{$)$}};
\draw [opacity = 0.25](15.8,2) rectangle (16.6,2.8);
\node at (16.2,2.4) {\scalebox{0.7}{$($}};
\draw [opacity = 0.25](16.6,2) rectangle (17.4,2.8);
\node at (17,2.4) {\scalebox{0.7}{$a$}};
\draw [opacity = 0.25](17.4,2) rectangle (18.2,2.8);
\node at (17.8,2.2) {\scalebox{0.7}{$,$}};
\draw [opacity = 0.25](18.2,2) rectangle (19,2.8);
\node at (18.6,2.4) {\scalebox{0.7}{$\begin{tikzpicture}
\draw[fill = jaune] (0,0) circle (0.1);
\end{tikzpicture}$}};
\draw [opacity = 0.25](19,2) rectangle (19.8,2.8);
\node at (19.4,2.4) {\scalebox{0.7}{$)$}};
\draw [opacity = 0.25](19.8,2) rectangle (20.6,2.8);
\draw [opacity = 0.25](20.6,2) rectangle (21.4,2.8);
\node at (22,2.3) {$\cdots$};
\node at (16,3.4) {\scalebox{0.7}{Input tape of $T$.}};
\end{scope}

\begin{scope}[shift={(-3.5,-3)}, scale = 1.2]
\draw [opacity = 0.25](11,2) rectangle (11.8,2.8);
\node at (11.4,2.4) {\scalebox{0.7}{$\vartriangleright$}};
\draw [opacity = 0.25](11.8,2) rectangle (12.6,2.8);
\draw [opacity = 0.25](12.6,2) rectangle (13.4,2.8);
\draw [opacity = 0.25](13.4,2) rectangle (14.2,2.8);
\draw [opacity = 0.25](14.2,2) rectangle (15,2.8);
\draw [opacity = 0.25](15,2) rectangle (15.8,2.8);
\draw [opacity = 0.25](15.8,2) rectangle (16.6,2.8);
\draw [opacity = 0.25](16.6,2) rectangle (17.4,2.8);
\draw [opacity = 0.25](17.4,2) rectangle (18.2,2.8);
\draw [opacity = 0.25](18.2,2) rectangle (19,2.8);
\draw [opacity = 0.25](19,2) rectangle (19.8,2.8);
\draw [opacity = 0.25](19.8,2) rectangle (20.6,2.8);
\draw [opacity = 0.25](20.6,2) rectangle (21.4,2.8);
\node at (22,2.3) {$\cdots$};
\node at (16,3.4) {\scalebox{0.7}{Working tape of $T$.}};
\end{scope}

\end{scope}


\begin{scope}[shift={(2,-2)}, scale = \scale]
\def \c{0.5}
\def \b{0.3}
\def \a{0.2}

\draw [-] (4,-4.5) to (-8,-4.5);
\draw [-] (4,-4.5) to (10-6/8,3.5-1);
\node at (-6,-3.5) {\scalebox{0.7}{Layer 5}};
\node at (2.2,-3.5) {\scalebox{0.7}{$\mathcal{M}_{AUX}$}};

\end{scope}

\begin{scope}[shift={(4,-4)}, scale = \scale]
\def \c{0.5}
\def \b{0.3}
\def \a{0.2}

\draw [-] (4,-4.5) to (-8,-4.5);
\draw [-] (4,-4.5) to (10-6/8,3.5-1);
\node at (-6,-3.5) {\scalebox{0.7}{Layer 4}};
\node at (2.0,-3.5) {\scalebox{0.7}{$\mathcal{M}_{ORACLE}$}};

\end{scope}


\begin{scope}[shift={(8,-8)}, scale = \scale]
\def \c{0.5}
\def \b{0.3}
\def \a{0.2}

\begin{scope}[opacity=0.25]
\draw [->] (0,0) to (-5,0);
\draw [->] (0,0) to (5,0);
\draw [->] (0,0) to (-3,-4);
\draw [->] (0,0) to (3,4);


\draw [->] (-3,0) to (-4,-4/3);
\draw [->] (-3,0) to (-2, 4/3);

\draw [->] (3,0) to (2,-4/3);
\draw [->] (3,0) to (4, 4/3);

\draw [->] (-2,-8/3) to (-11/3,-8/3);
\draw [->] (-2,-8/3) to (-1/3,-8/3);

\draw [->] (2,8/3) to (11/3,8/3);
\draw [->] (2,8/3) to (1/3,8/3);


\draw [->] (-4,0) to (-3.5,4/6);
\draw [->] (-4,0) to (-4.5,-4/6);
\draw [->] (-3.5,-2/3) to (-4.3,-2/3);
\draw [->] (-3.5,-2/3) to (-2.7,-2/3);
\draw [->] (-2.5,2/3) to (-3.3,2/3);
\draw [->] (-2.5,2/3) to (-1.7,2/3);

\draw [->] (4,0) to (3.5,-4/6);
\draw [->] (4,0) to (4.5,4/6);
\draw [->] (3.5,2/3) to (4.3,2/3);
\draw [->] (3.5,2/3) to (2.7,2/3);
\draw [->] (2.5,-2/3) to (3.3,-2/3);
\draw [->] (2.5,-2/3) to (1.7,-2/3);

\draw [->] (-2.5,-10/3) to (-3.3,-10/3);
\draw [->] (-2.5,-10/3) to (-1.7,-10/3);
\draw [->] (-3,-8/3) to (-3.5,-10/3);
\draw [->] (-3,-8/3) to (-2.5,-6/3);
\draw [->] (-1,-8/3) to (-1.5,-10/3);
\draw [->] (-1,-8/3) to (-0.5,-6/3);

\draw [->] (2.5,10/3) to (3.3,10/3);
\draw [->] (2.5,10/3) to (1.7,10/3);
\draw [->] (3,8/3) to (3.5,10/3);
\draw [->] (3,8/3) to (2.5,6/3);
\draw [->] (1,8/3) to (1.5,10/3);
\draw [->] (1,8/3) to (0.5,6/3);


\end{scope}

\begin{scope}[scale = 0.8, opacity = 0.1]
\draw[fill = black, rotate=315] (0, 0) ellipse (3cm and 4.2cm);
\end{scope}

\draw [->] (0,0) to (-3,0);
\draw [->] (0,0) to (3,0);
\draw [->] (0,0) to (-2,-8/3);
\draw [->] (0,0) to (2,8/3);

\node at (0,0) {$\vartriangleright$};
\node at (2,8/3) {{$\otimes$}};
\node at (2,1) {\scalebox{0.7}{$B_n$}};

\path
(3, 2) edge [dashed, bend left=30,->] node[swap]  {} (9,5);

\draw [-] (4,-4.5) to (-8,-4.5);
\draw [-] (4,-4.5) to (10,3.5);
\node at (-6,-3.5) {\scalebox{0.7}{Layer 3}};
\node at (2.5,-3.5) {\scalebox{0.7}{$\mathcal{M}_{VISIT}$}};

\end{scope}


\begin{scope}[shift={(12,-12)}, scale = \scale]
\def \c{0.5}
\def \b{0.3}
\def \a{0.2}

\begin{scope}[opacity=0.25]
\draw [->] (0,0) to (-5,0);
\draw [->] (0,0) to (5,0);
\draw [->] (0,0) to (-3,-4);
\draw [->] (0,0) to (3,4);


\draw [->] (-3,0) to (-4,-4/3);
\draw [->] (-3,0) to (-2, 4/3);

\draw [->] (3,0) to (2,-4/3);
\draw [->] (3,0) to (4, 4/3);

\draw [->] (-2,-8/3) to (-11/3,-8/3);
\draw [->] (-2,-8/3) to (-1/3,-8/3);

\draw [->] (2,8/3) to (11/3,8/3);
\draw [->] (2,8/3) to (1/3,8/3);


\draw [->] (-4,0) to (-3.5,4/6);
\draw [->] (-4,0) to (-4.5,-4/6);
\draw [->] (-3.5,-2/3) to (-4.3,-2/3);
\draw [->] (-3.5,-2/3) to (-2.7,-2/3);
\draw [->] (-2.5,2/3) to (-3.3,2/3);
\draw [->] (-2.5,2/3) to (-1.7,2/3);

\draw [->] (4,0) to (3.5,-4/6);
\draw [->] (4,0) to (4.5,4/6);
\draw [->] (3.5,2/3) to (4.3,2/3);
\draw [->] (3.5,2/3) to (2.7,2/3);
\draw [->] (2.5,-2/3) to (3.3,-2/3);
\draw [->] (2.5,-2/3) to (1.7,-2/3);

\draw [->] (-2.5,-10/3) to (-3.3,-10/3);
\draw [->] (-2.5,-10/3) to (-1.7,-10/3);
\draw [->] (-3,-8/3) to (-3.5,-10/3);
\draw [->] (-3,-8/3) to (-2.5,-6/3);
\draw [->] (-1,-8/3) to (-1.5,-10/3);
\draw [->] (-1,-8/3) to (-0.5,-6/3);

\draw [->] (2.5,10/3) to (3.3,10/3);
\draw [->] (2.5,10/3) to (1.7,10/3);
\draw [->] (3,8/3) to (3.5,10/3);
\draw [->] (3,8/3) to (2.5,6/3);
\draw [->] (1,8/3) to (1.5,10/3);
\draw [->] (1,8/3) to (0.5,6/3);

\end{scope}

\node at (0,0) {\large$\vartriangleright$};
\draw [-] (0.25,0) to (2.75,0);
\node at (3,0) {$\otimes$};
\draw [-] (3.25,0) to (3.75,0);
\node at (4,0) {$\otimes$};
\draw [-] (4.25,0) to (4.75,0);
\node at (5,0) {$\otimes$};
\draw [-] (5+3/16,0.25) to (5+9/16,0.75);
\node at (5+3/4,1) {$\otimes$};
\draw [-] (5.25+3/4,1) to (5.75+3/4,1);
\node at (6+3/4,1) {$\otimes$};

\draw [-] (4,-4.5) to (-8,-4.5);
\draw [-] (4,-4.5) to (10,3.5);

\node at (-6,-3.5) {\scalebox{0.7}{Layer 2}};
\node at (2.5,-3.5) {\scalebox{0.7}{$\mathcal{M}_{PATH}$}};

\end{scope}


\begin{scope}[shift={(16,-16)}, scale = \scale]
\def \c{0.5}
\def \b{0.3}
\def \a{0.2}

\begin{scope}[opacity=0.25]
\draw [->] (0,0) to (-5,0);
\draw [->] (0,0) to (5,0);
\draw [->] (0,0) to (-3,-4);
\draw [->] (0,0) to (3,4);


\draw [->] (-3,0) to (-4,-4/3);
\draw [->] (-3,0) to (-2, 4/3);

\draw [->] (3,0) to (2,-4/3);
\draw [->] (3,0) to (4, 4/3);

\draw [->] (-2,-8/3) to (-11/3,-8/3);
\draw [->] (-2,-8/3) to (-1/3,-8/3);

\draw [->] (2,8/3) to (11/3,8/3);
\draw [->] (2,8/3) to (1/3,8/3);

\draw [->] (-4,0) to (-3.5,4/6);
\draw [->] (-4,0) to (-4.5,-4/6);
\draw [->] (-3.5,-2/3) to (-4.3,-2/3);
\draw [->] (-3.5,-2/3) to (-2.7,-2/3);
\draw [->] (-2.5,2/3) to (-3.3,2/3);
\draw [->] (-2.5,2/3) to (-1.7,2/3);

\draw [->] (4,0) to (3.5,-4/6);
\draw [->] (4,0) to (4.5,4/6);
\draw [->] (3.5,2/3) to (4.3,2/3);
\draw [->] (3.5,2/3) to (2.7,2/3);
\draw [->] (2.5,-2/3) to (3.3,-2/3);
\draw [->] (2.5,-2/3) to (1.7,-2/3);

\draw [->] (-2.5,-10/3) to (-3.3,-10/3);
\draw [->] (-2.5,-10/3) to (-1.7,-10/3);
\draw [->] (-3,-8/3) to (-3.5,-10/3);
\draw [->] (-3,-8/3) to (-2.5,-6/3);
\draw [->] (-1,-8/3) to (-1.5,-10/3);
\draw [->] (-1,-8/3) to (-0.5,-6/3);

\draw [->] (2.5,10/3) to (3.3,10/3);
\draw [->] (2.5,10/3) to (1.7,10/3);
\draw [->] (3,8/3) to (3.5,10/3);
\draw [->] (3,8/3) to (2.5,6/3);
\draw [->] (1,8/3) to (1.5,10/3);
\draw [->] (1,8/3) to (0.5,6/3);

\end{scope}

 \draw[preaction={fill, blue!25},pattern=north west lines] (0,0) circle (\c);
 \draw[fill = red] (-3,0) circle (\b);
 \draw[fill = jaune] (3,0) circle (\b);
 \draw[fill = vert] (2,8/3) circle (\b);
 \draw[fill = purple] (-2,-8/3) circle (\b);

\node at (2,1) {\scalebox{0.7}{$p$}};
\draw [-] (4,-4.5) to (-8,-4.5);
\draw [-] (4,-4.5) to (10,3.5);

\node at (-6,-3.5) {\scalebox{0.7}{Layer 1}};
\node at (2.5,-3.5) {\scalebox{0.7}{$\mathcal{M}_{STORE}$}};

\end{scope}


\begin{scope}[shift={(16,-0.3)}, scale = 0.4]

\begin{scope}[opacity=0.25]
\draw [->] (0,0) to (-5,0);
\draw [->] (0,0) to (5,0);
\draw [->] (0,0) to (-3,-4);
\draw [->] (0,0) to (3,4);


\draw [->] (-3,0) to (-4,-4/3);
\draw [->] (-3,0) to (-2, 4/3);

\draw [->] (3,0) to (2,-4/3);
\draw [->] (3,0) to (4, 4/3);

\draw [->] (-2,-8/3) to (-11/3,-8/3);
\draw [->] (-2,-8/3) to (-1/3,-8/3);

\draw [->] (2,8/3) to (11/3,8/3);
\draw [->] (2,8/3) to (1/3,8/3);

\draw [->] (-4,0) to (-3.5,4/6);
\draw [->] (-4,0) to (-4.5,-4/6);
\draw [->] (-3.5,-2/3) to (-4.3,-2/3);
\draw [->] (-3.5,-2/3) to (-2.7,-2/3);
\draw [->] (-2.5,2/3) to (-3.3,2/3);
\draw [->] (-2.5,2/3) to (-1.7,2/3);

\draw [->] (4,0) to (3.5,-4/6);
\draw [->] (4,0) to (4.5,4/6);
\draw [->] (3.5,2/3) to (4.3,2/3);
\draw [->] (3.5,2/3) to (2.7,2/3);
\draw [->] (2.5,-2/3) to (3.3,-2/3);
\draw [->] (2.5,-2/3) to (1.7,-2/3);

\draw [->] (-2.5,-10/3) to (-3.3,-10/3);
\draw [->] (-2.5,-10/3) to (-1.7,-10/3);
\draw [->] (-3,-8/3) to (-3.5,-10/3);
\draw [->] (-3,-8/3) to (-2.5,-6/3);
\draw [->] (-1,-8/3) to (-1.5,-10/3);
\draw [->] (-1,-8/3) to (-0.5,-6/3);

\draw [->] (2.5,10/3) to (3.3,10/3);
\draw [->] (2.5,10/3) to (1.7,10/3);
\draw [->] (3,8/3) to (3.5,10/3);
\draw [->] (3,8/3) to (2.5,6/3);
\draw [->] (1,8/3) to (1.5,10/3);
\draw [->] (1,8/3) to (0.5,6/3);

\end{scope}

\draw [-] (4,-4.5) to (-8,-4.5);
\draw [-] (4,-4.5) to (10,3.5);
\node at (-6,-3.5) {\scalebox{0.5}{Layer 3.1}};
\node at (2.5,-3.5) {\scalebox{0.5}{$\mathcal{M}_{PATH}$}};

\node at (0,0) {$\vartriangleright$};
\draw [-] (0.25,0) to (2.75,0);
\node at (3,0) {\scalebox{0.5}{$\otimes$}};
\draw [-] (3.25,0) to (3.75,0);
\node at (4,0) {\scalebox{0.5}{$\otimes$}};
\draw [-] (4.25,0) to (4.75,0);
\node at (5,0) {\scalebox{0.5}{$\otimes$}};
\draw [-] (5+3/16,0.25) to (5+9/16,0.75);
\node at (5+3/4,1) {\scalebox{0.5}{$\otimes$}};
\draw [-] (5.25+3/4,1) to (5.75+3/4,1);
\node at (6+3/4,1) {\scalebox{0.5}{$\otimes$}};

\end{scope}

\begin{scope}[shift={(18,-3)}, scale = 0.4]

\begin{scope}[opacity=0.25]
\draw [->] (0,0) to (-5,0);
\draw [->] (0,0) to (5,0);
\draw [->] (0,0) to (-3,-4);
\draw [->] (0,0) to (3,4);


\draw [->] (-3,0) to (-4,-4/3);
\draw [->] (-3,0) to (-2, 4/3);

\draw [->] (3,0) to (2,-4/3);
\draw [->] (3,0) to (4, 4/3);

\draw [->] (-2,-8/3) to (-11/3,-8/3);
\draw [->] (-2,-8/3) to (-1/3,-8/3);

\draw [->] (2,8/3) to (11/3,8/3);
\draw [->] (2,8/3) to (1/3,8/3);

\draw [->] (-4,0) to (-3.5,4/6);
\draw [->] (-4,0) to (-4.5,-4/6);
\draw [->] (-3.5,-2/3) to (-4.3,-2/3);
\draw [->] (-3.5,-2/3) to (-2.7,-2/3);
\draw [->] (-2.5,2/3) to (-3.3,2/3);
\draw [->] (-2.5,2/3) to (-1.7,2/3);

\draw [->] (4,0) to (3.5,-4/6);
\draw [->] (4,0) to (4.5,4/6);
\draw [->] (3.5,2/3) to (4.3,2/3);
\draw [->] (3.5,2/3) to (2.7,2/3);
\draw [->] (2.5,-2/3) to (3.3,-2/3);
\draw [->] (2.5,-2/3) to (1.7,-2/3);

\draw [->] (-2.5,-10/3) to (-3.3,-10/3);
\draw [->] (-2.5,-10/3) to (-1.7,-10/3);
\draw [->] (-3,-8/3) to (-3.5,-10/3);
\draw [->] (-3,-8/3) to (-2.5,-6/3);
\draw [->] (-1,-8/3) to (-1.5,-10/3);
\draw [->] (-1,-8/3) to (-0.5,-6/3);

\draw [->] (2.5,10/3) to (3.3,10/3);
\draw [->] (2.5,10/3) to (1.7,10/3);
\draw [->] (3,8/3) to (3.5,10/3);
\draw [->] (3,8/3) to (2.5,6/3);
\draw [->] (1,8/3) to (1.5,10/3);
\draw [->] (1,8/3) to (0.5,6/3);

\end{scope}

\draw [-] (4,-4.5) to (-8,-4.5);
\draw [-] (4,-4.5) to (10,3.5);
\node at (-6,-3.5) {\scalebox{0.5}{Layer 3.2}};
\node at (2.5,-3.5) {\scalebox{0.5}{Counter}};

\node at (0,0) {$\vartriangleright$};
\draw [-] (0.25,0) to (5,0);
\draw [-] (5,0) to (5.5, 4/6);
\draw [->] (5.5, 4/6) to (6.5, 4/6);
\node at (6,4/6+1) {\scalebox{0.5}{$n = 1$}};

\node at (-10,-4.5) {\rotatebox[]{35}{\scalebox{2}[8]{$\{$}}};

\end{scope}

\begin{scope}[shift={(20,-5.7)}, scale = 0.4]

\begin{scope}[opacity=0.25]
\draw [->] (0,0) to (-5,0);
\draw [->] (0,0) to (5,0);
\draw [->] (0,0) to (-3,-4);
\draw [->] (0,0) to (3,4);


\draw [->] (-3,0) to (-4,-4/3);
\draw [->] (-3,0) to (-2, 4/3);

\draw [->] (3,0) to (2,-4/3);
\draw [->] (3,0) to (4, 4/3);

\draw [->] (-2,-8/3) to (-11/3,-8/3);
\draw [->] (-2,-8/3) to (-1/3,-8/3);

\draw [->] (2,8/3) to (11/3,8/3);
\draw [->] (2,8/3) to (1/3,8/3);

\draw [->] (-4,0) to (-3.5,4/6);
\draw [->] (-4,0) to (-4.5,-4/6);
\draw [->] (-3.5,-2/3) to (-4.3,-2/3);
\draw [->] (-3.5,-2/3) to (-2.7,-2/3);
\draw [->] (-2.5,2/3) to (-3.3,2/3);
\draw [->] (-2.5,2/3) to (-1.7,2/3);

\draw [->] (4,0) to (3.5,-4/6);
\draw [->] (4,0) to (4.5,4/6);
\draw [->] (3.5,2/3) to (4.3,2/3);
\draw [->] (3.5,2/3) to (2.7,2/3);
\draw [->] (2.5,-2/3) to (3.3,-2/3);
\draw [->] (2.5,-2/3) to (1.7,-2/3);

\draw [->] (-2.5,-10/3) to (-3.3,-10/3);
\draw [->] (-2.5,-10/3) to (-1.7,-10/3);
\draw [->] (-3,-8/3) to (-3.5,-10/3);
\draw [->] (-3,-8/3) to (-2.5,-6/3);
\draw [->] (-1,-8/3) to (-1.5,-10/3);
\draw [->] (-1,-8/3) to (-0.5,-6/3);

\draw [->] (2.5,10/3) to (3.3,10/3);
\draw [->] (2.5,10/3) to (1.7,10/3);
\draw [->] (3,8/3) to (3.5,10/3);
\draw [->] (3,8/3) to (2.5,6/3);
\draw [->] (1,8/3) to (1.5,10/3);
\draw [->] (1,8/3) to (0.5,6/3);

\end{scope}

\draw [->] (0,0) to (-3,0);
\draw [->] (0,0) to (3,0);
\draw [->] (0,0) to (-2,-8/3);
\draw [->] (0,0) to (2,8/3);

\begin{scope}[scale = 0.8, opacity = 0.1]
\draw[fill = black, rotate=315] (0, 0) ellipse (3cm and 4.2cm);
\end{scope}

\node at (0,0) {$\vartriangleright$};
\node at (2,8/3) {\scalebox{0.5}{$\otimes$}};
\node at (2,1) {\scalebox{0.5}{$B_n$}};

\draw [-] (4,-4.5) to (-8,-4.5);
\draw [-] (4,-4.5) to (10,3.5);
\node at (-6,-3.5) {\scalebox{0.5}{Layer 3.3}};
\node at (2.5,-3.5) {\scalebox{0.5}{$\mathcal{M}'_{PATH}$}};

\end{scope}

\end{tikzpicture}

%% file: example_yn.tex
\begin{tikzpicture}[scale=0.25]

\clip[draw,decorate,decoration={random steps, segment length=3pt, amplitude=1pt}] (0.5,0.5) rectangle (39.5,16.5); 

\foreach \y in {0,...,40} {
	\foreach \yy in {0,...,20} {
		\lettre{\y}{\yy}{black!70}
	}
}
\foreach \a/\b in {2/3, 5/8, 8/1,15/2, 23/2, 12/7, 20/7, 26/9, 34/9, 30/3, 36/2, -2/8, 3/14, 9/12, 17/12, 23/14, 30/13, 37/14, 13/16} {
		
\lettre{\a}{\b}{white}
\lettre{\a+1}{\b}{vert}
\lettre{\a-1}{\b}{vert}
\lettre{\a}{\b+1}{vert}
\lettre{\a}{\b-1}{vert}
\lettre{\a+2}{\b}{vert}
\lettre{\a-2}{\b}{vert}
\lettre{\a}{\b+2}{vert}
\lettre{\a}{\b-2}{vert}
\lettre{\a+1}{\b+1}{vert}
\lettre{\a+1}{\b-1}{vert}
\lettre{\a-1}{\b+1}{vert}
\lettre{\a-1}{\b-1}{vert}
}
\lettre{14}{12}{vert}
\lettre{13}{12}{vert}
\lettre{14}{11}{vert}

\lettre{17}{7}{vert}
\lettre{17}{6}{vert}
\lettre{18}{6}{vert}

\lettre{30}{7}{vert}
\end{tikzpicture}

%% file: patate.tex
\scalebox{0.6}{

\begin{tikzpicture}[scale=1]

\begin{scope}[shift={(0,0)}]

\draw[thick] (-6.25,-0.25) rectangle +(0.5,0.5);
\node at (-6,0) {$\oplus$};
\node at (6.5,0) {$y_n \in Y_n$};

\clip[draw,decorate,decoration={random steps, segment length=3pt, amplitude=1pt}] (0,0) ellipse (5 and 1);
\draw[fill = black!70] (0,0) ellipse (5.2 and 1.1);
\draw[fill = vert, decorate,decoration={random steps, segment length=3pt, amplitude=1pt}] (0,0) ellipse (1 and 0.4);
\draw[fill = white] (0,0) ellipse (0.15 and 0.1);

\draw[fill = vert, decorate,decoration={random steps, segment length=3pt, amplitude=1pt}] (2,1) ellipse (1 and 0.3);

\draw[fill = vert, decorate,decoration={random steps, segment length=3pt, amplitude=1pt}] (-1.3,-0.8) ellipse (1 and 0.3);
\draw[fill = white] (-1.3,-0.8) ellipse (0.15 and 0.1);

\draw[fill = vert, decorate,decoration={random steps, segment length=3pt, amplitude=1pt}] (-2.5,0.5) ellipse (1 and 0.3);
\draw[fill = white] (-2.5,0.5) ellipse (0.15 and 0.1);

\draw[fill = vert, decorate,decoration={random steps, segment length=3pt, amplitude=1pt}] (-3.5,-0.3) ellipse (1 and 0.3);
\draw[fill = white] (-3.5,-0.3) ellipse (0.15 and 0.1);

\draw[fill = vert, decorate,decoration={random steps, segment length=3pt, amplitude=1pt}] (2.5,-0.1) ellipse (1 and 0.35);
\draw[fill = white] (2.5,-0.1) ellipse (0.15 and 0.1);

\draw[fill = vert, decorate,decoration={random steps, segment length=3pt, amplitude=1pt}] (4.5,0.2) ellipse (1 and 0.35);
\draw[fill = white] (4.5,0.2) ellipse (0.15 and 0.1);

\draw[fill = vert, decorate,decoration={random steps, segment length=3pt, amplitude=1pt}] (1,-1) ellipse (1 and 0.35);
\draw[fill = white] (1,-1) ellipse (0.15 and 0.1);

\draw[fill = vert, decorate,decoration={random steps, segment length=3pt, amplitude=1pt}] (-0.6,1) ellipse (1 and 0.35);
\draw[fill = white] (-0.6,1) ellipse (0.15 and 0.1);
\end{scope}

\draw[->, thick] (0,0) to (0,2.5);
\begin{scope}[shift={(0,2.5)}]
\draw[thick] (-6.25,-0.25) rectangle +(0.5,0.5);
\node at (-6,0) {$\vartriangleright$};
\node at (6.5,0) {$y_n$};
\clip[draw,decorate,decoration={random steps, segment length=3pt, amplitude=1pt}] (0,0) ellipse (5 and 1);
\draw[fill = black!70] (0,0) ellipse (5.2 and 1.1);
\draw[fill = vert, decorate,decoration={random steps, segment length=3pt, amplitude=1pt}] (0,0) ellipse (1 and 0.4);
\draw[fill = white] (0,0) ellipse (0.15 and 0.1);

\draw[fill = vert, decorate,decoration={random steps, segment length=3pt, amplitude=1pt}] (2,1) ellipse (1 and 0.3);

\draw[fill = vert, decorate,decoration={random steps, segment length=3pt, amplitude=1pt}] (-1.3,-0.8) ellipse (1 and 0.3);
\draw[fill = white] (-1.3,-0.8) ellipse (0.15 and 0.1);

\draw[fill = vert, decorate,decoration={random steps, segment length=3pt, amplitude=1pt}] (-2.5,0.5) ellipse (1 and 0.3);
\draw[fill = white] (-2.5,0.5) ellipse (0.15 and 0.1);

\draw[fill = vert, decorate,decoration={random steps, segment length=3pt, amplitude=1pt}] (-3.5,-0.3) ellipse (1 and 0.3);
\draw[fill = white] (-3.5,-0.3) ellipse (0.15 and 0.1);

\draw[fill = vert, decorate,decoration={random steps, segment length=3pt, amplitude=1pt}] (2.5,-0.1) ellipse (1 and 0.35);
\draw[fill = white] (2.5,-0.1) ellipse (0.15 and 0.1);

\draw[fill = vert, decorate,decoration={random steps, segment length=3pt, amplitude=1pt}] (4.5,0.2) ellipse (1 and 0.35);
\draw[fill = white] (4.5,0.2) ellipse (0.15 and 0.1);

\draw[fill = vert, decorate,decoration={random steps, segment length=3pt, amplitude=1pt}] (1,-1) ellipse (1 and 0.35);
\draw[fill = white] (1,-1) ellipse (0.15 and 0.1);

\draw[fill = vert, decorate,decoration={random steps, segment length=3pt, amplitude=1pt}] (-0.6,1) ellipse (1 and 0.35);
\draw[fill = white] (-0.6,1) ellipse (0.15 and 0.1);
\end{scope}
\draw[->, thick] (0,2.5) to (0,5);

\begin{scope}[shift={(0,5)}]
\draw[thick] (-6.25,-0.25) rectangle +(0.5,0.5);
\node at (-6,0) {$s$};
\node at (6.5,0) {$\sigma_s(y_n)$};
\clip[draw,decorate,decoration={random steps, segment length=3pt, amplitude=1pt}] (0,0) ellipse (5 and 1);
\draw[fill = black!70] (0,0) ellipse (5.2 and 1.1);
\begin{scope}[shift={(0.5,0.4)}]
\draw[fill = vert, decorate,decoration={random steps, segment length=3pt, amplitude=1pt}] (0,0) ellipse (1 and 0.4);
\draw[fill = white] (0,0) ellipse (0.15 and 0.1);

\draw[fill = vert, decorate,decoration={random steps, segment length=3pt, amplitude=1pt}] (2,1) ellipse (1 and 0.3);

\draw[fill = vert, decorate,decoration={random steps, segment length=3pt, amplitude=1pt}] (-1.3,-0.8) ellipse (1 and 0.3);
\draw[fill = white] (-1.3,-0.8) ellipse (0.15 and 0.1);

\draw[fill = vert, decorate,decoration={random steps, segment length=3pt, amplitude=1pt}] (-2.5,0.5) ellipse (1 and 0.3);
\draw[fill = white] (-2.5,0.5) ellipse (0.15 and 0.1);

\draw[fill = vert, decorate,decoration={random steps, segment length=3pt, amplitude=1pt}] (-3.5,-0.3) ellipse (1 and 0.3);
\draw[fill = white] (-3.5,-0.3) ellipse (0.15 and 0.1);

\draw[fill = vert, decorate,decoration={random steps, segment length=3pt, amplitude=1pt}] (2.5,-0.1) ellipse (1 and 0.35);
\draw[fill = white] (2.5,-0.1) ellipse (0.15 and 0.1);

\draw[fill = vert, decorate,decoration={random steps, segment length=3pt, amplitude=1pt}] (4.5,0.2) ellipse (1 and 0.35);
\draw[fill = white] (4.5,0.2) ellipse (0.15 and 0.1);

\draw[fill = vert, decorate,decoration={random steps, segment length=3pt, amplitude=1pt}] (1,-1) ellipse (1 and 0.35);
\draw[fill = white] (1,-1) ellipse (0.15 and 0.1);

\draw[fill = vert, decorate,decoration={random steps, segment length=3pt, amplitude=1pt}] (-0.6,1) ellipse (1 and 0.35);
\draw[fill = white] (-0.6,1) ellipse (0.15 and 0.1);
\end{scope}
\end{scope}
\draw[->, thick] (0,5) to (0,7.5);
\begin{scope}[shift={(0,7.5)}]
\draw[thick] (-6.25,-0.25) rectangle +(0.5,0.5);
\node at (-6,0) {$\bullet$};
\node at (6.5,0) {$\sigma_s(y_n)$};
\clip[draw,decorate,decoration={random steps, segment length=3pt, amplitude=1pt}] (0,0) ellipse (5 and 1);
\draw[fill = black!70] (0,0) ellipse (5.2 and 1.1);
\begin{scope}[shift={(0.5,0.4)}]
\draw[fill = vert, decorate,decoration={random steps, segment length=3pt, amplitude=1pt}] (0,0) ellipse (1 and 0.4);
\draw[fill = white] (0,0) ellipse (0.15 and 0.1);

\draw[fill = vert, decorate,decoration={random steps, segment length=3pt, amplitude=1pt}] (2,1) ellipse (1 and 0.3);

\draw[fill = vert, decorate,decoration={random steps, segment length=3pt, amplitude=1pt}] (-1.3,-0.8) ellipse (1 and 0.3);
\draw[fill = white] (-1.3,-0.8) ellipse (0.15 and 0.1);

\draw[fill = vert, decorate,decoration={random steps, segment length=3pt, amplitude=1pt}] (-2.5,0.5) ellipse (1 and 0.3);
\draw[fill = white] (-2.5,0.5) ellipse (0.15 and 0.1);

\draw[fill = vert, decorate,decoration={random steps, segment length=3pt, amplitude=1pt}] (-3.5,-0.3) ellipse (1 and 0.3);
\draw[fill = white] (-3.5,-0.3) ellipse (0.15 and 0.1);

\draw[fill = vert, decorate,decoration={random steps, segment length=3pt, amplitude=1pt}] (2.5,-0.1) ellipse (1 and 0.35);
\draw[fill = white] (2.5,-0.1) ellipse (0.15 and 0.1);

\draw[fill = vert, decorate,decoration={random steps, segment length=3pt, amplitude=1pt}] (4.5,0.2) ellipse (1 and 0.35);
\draw[fill = white] (4.5,0.2) ellipse (0.15 and 0.1);

\draw[fill = vert, decorate,decoration={random steps, segment length=3pt, amplitude=1pt}] (1,-1) ellipse (1 and 0.35);
\draw[fill = white] (1,-1) ellipse (0.15 and 0.1);

\draw[fill = vert, decorate,decoration={random steps, segment length=3pt, amplitude=1pt}] (-0.6,1) ellipse (1 and 0.35);
\draw[fill = white] (-0.6,1) ellipse (0.15 and 0.1);
\end{scope}
\end{scope}

\draw[->, thick] (0,7.5) to (0,10);
\begin{scope}[shift={(0,10)}]
\draw[thick] (-6.25,-0.25) rectangle +(0.5,0.5);
\node at (-6,0) {$s$};
\node at (-5.82,0.05) {${\tiny ^{1}}$};
\draw[-] (-6,0.15) to(-5.9,0.15);

\node at (6.5,0) {$y_n$};
\clip[draw,decorate,decoration={random steps, segment length=3pt, amplitude=1pt}] (0,0) ellipse (5 and 1);
\draw[fill = black!70] (0,0) ellipse (5.2 and 1.1);
\draw[fill = vert, decorate,decoration={random steps, segment length=3pt, amplitude=1pt}] (0,0) ellipse (1 and 0.4);
\draw[fill = white] (0,0) ellipse (0.15 and 0.1);

\draw[fill = vert, decorate,decoration={random steps, segment length=3pt, amplitude=1pt}] (2,1) ellipse (1 and 0.3);

\draw[fill = vert, decorate,decoration={random steps, segment length=3pt, amplitude=1pt}] (-1.3,-0.8) ellipse (1 and 0.3);
\draw[fill = white] (-1.3,-0.8) ellipse (0.15 and 0.1);

\draw[fill = vert, decorate,decoration={random steps, segment length=3pt, amplitude=1pt}] (-2.5,0.5) ellipse (1 and 0.3);
\draw[fill = white] (-2.5,0.5) ellipse (0.15 and 0.1);

\draw[fill = vert, decorate,decoration={random steps, segment length=3pt, amplitude=1pt}] (-3.5,-0.3) ellipse (1 and 0.3);
\draw[fill = white] (-3.5,-0.3) ellipse (0.15 and 0.1);

\draw[fill = vert, decorate,decoration={random steps, segment length=3pt, amplitude=1pt}] (2.5,-0.1) ellipse (1 and 0.35);
\draw[fill = white] (2.5,-0.1) ellipse (0.15 and 0.1);

\draw[fill = vert, decorate,decoration={random steps, segment length=3pt, amplitude=1pt}] (4.5,0.2) ellipse (1 and 0.35);
\draw[fill = white] (4.5,0.2) ellipse (0.15 and 0.1);

\draw[fill = vert, decorate,decoration={random steps, segment length=3pt, amplitude=1pt}] (1,-1) ellipse (1 and 0.35);
\draw[fill = white] (1,-1) ellipse (0.15 and 0.1);

\draw[fill = vert, decorate,decoration={random steps, segment length=3pt, amplitude=1pt}] (-0.6,1) ellipse (1 and 0.35);
\draw[fill = white] (-0.6,1) ellipse (0.15 and 0.1);
\end{scope}

\draw[->, thick] (0,10) to (0,12.5);
\begin{scope}[shift={(0,12.5)}]
\draw[thick] (-6.25,-0.25) rectangle +(0.5,0.5);
\node at (-6,0) {$\oplus$};
\node at (6.5,0) {$y_{n+1} \in Y_{n+1}$};
\clip[draw,decorate,decoration={random steps, segment length=3pt, amplitude=1pt}] (0,0) ellipse (5 and 1);
\draw[fill = black!70] (0,0) ellipse (5.2 and 1.1);

\draw[fill = vert, decorate,decoration={random steps, segment length=3pt, amplitude=1pt}] (0,0) ellipse (1.5 and 0.6);
\draw[fill = white] (0,0) ellipse (0.15 and 0.1);

\draw[fill = vert, decorate,decoration={random steps, segment length=3pt, amplitude=1pt}] (-3.5,0.4) ellipse (1.2 and 0.5);
\draw[fill = white] (-3.5,0.4) ellipse (0.15 and 0.1);

\draw[fill = vert, decorate,decoration={random steps, segment length=3pt, amplitude=1pt}] (3,0.6) ellipse (1.2 and 0.5);
\draw[fill = white] (3,0.6) ellipse (0.15 and 0.1);

\draw[fill = vert, decorate,decoration={random steps, segment length=3pt, amplitude=1pt}] (-2,-1.2) ellipse (1.2 and 0.5);
\draw[fill = white] (-2,-1.2) ellipse (0.15 and 0.1);

\draw[fill = vert, decorate,decoration={random steps, segment length=3pt, amplitude=1pt}] (4.5,-0.5) ellipse (1.2 and 0.5);
\draw[fill = white] (4.5,-0.5) ellipse (0.15 and 0.1);
\end{scope}
\draw[->, thick] (0,12.5) to (0,14);
\draw[thick] (-6.25,-2.75) rectangle +(0.5,0.5);
\node at (-6,-1.25) {$\vdots$};
\node at (-6,-2.5) {$\star$};
\end{tikzpicture}
}

%% file: example_rb.tex
\begin{tikzpicture}[scale=0.25]

\clip[draw,decorate,decoration={random steps, segment length=3pt, amplitude=1pt}] (0.5,2.5) rectangle (39.5,11.5); 

\foreach \yy in {0,...,5} {
\foreach \y in {0,...,20} {
	\lettre{2*\y}{\yy}{rouge}
	\lettre{2*\y+1}{\yy}{bleu}
}
}

	\foreach \y in {0,...,10} {
		\foreach \yy in {0,1} {
		\lettre{4*\y+\yy}{6}{rouge}
		\lettre{4*\y+2+\yy}{6}{bleu}
	}
	}
	\foreach \y in {0,...,5} {
		\foreach \yy in {0,...,3} {
			\lettre{8*\y+\yy}{7}{rouge}
			\lettre{8*\y+4+\yy}{7}{bleu}
		}
	}
	\foreach \y in {0,...,3} {
		\foreach \yy in {0,...,7} {
			\lettre{16*\y+\yy}{8}{rouge}
			\lettre{16*\y+8+\yy}{8}{bleu}
		}
	}
	\foreach \y in {0,1} {
		\foreach \yy in {0,...,15} {
			\lettre{32*\y+\yy}{9}{rouge}
			\lettre{32*\y+16+\yy}{9}{bleu}
		}
	}
		\foreach \yy in {0,...,31} {
			\lettre{\yy}{10}{rouge}
			\lettre{32+\yy}{10}{bleu}
		}
	\foreach \yy in {11,12,13} {
		\foreach \y in {0,...,40} {
			\lettre{\y}{\yy}{rouge}
		}
	}

\end{tikzpicture}